\documentclass[preprint,12pt,authoryear]{elsarticle}

\usepackage[english]{babel}
\usepackage{latexsym,graphics,epsfig,subfigure,color,lineno}
\usepackage{amsfonts,amsmath,amssymb,amsthm}
\usepackage{mathtools,xparse}
\usepackage{upgreek}
\usepackage[normalem]{ulem}
\usepackage{pifont}
\usepackage[margin=3.5cm]{geometry}

\input{defs.inp}

\def\lambdaexp{\frac{1}{2\lambda}} 

\DeclarePairedDelimiter{\abs}{\lvert}{\rvert}
\DeclarePairedDelimiter{\norm}{\lVert}{\rVert}

\def\Pr{\mathop{\rm Pr}\nolimits}

\DeclareMathOperator{\MVN}{MVN}

\newcommand{\tr}{\mbox{\rm tr}}

\newcommand{\fbm}{{FBM}}
\newcommand{\LL}{L} 
\newcommand*\diff{\mathop{}\!\mathrm{d}}

\newcommand{\inn}{\langle\cdot ,\cdot\rangle}

\usepackage{amssymb}

\journal{Econometrics and Statistics}

\begin{document}
	
	\begin{frontmatter}
		
		
		
		\title{Regression with I-priors}
		
		
		\author{Wicher P.\ Bergsma}
		\address{London School of Economics, Houghton Street, London, WC2A 2AE, United Kingdom}
		
		\begin{abstract}
The problem of estimating a parametric or nonparametric regression function in a model with normal errors is considered. 
For this purpose, a novel objective prior for the regression function is proposed, defined as the distribution maximizing entropy subject to a suitable constraint based on the Fisher information on the regression function. The prior is named I-prior. For the present model, it is Gaussian with covariance kernel proportional to the Fisher information, and mean chosen a priori (e.g., 0).

The I-prior has the intuitively appealing property that the more information is available about a linear functional of the regression function, the larger its prior variance, and, broadly speaking, the less influential the prior is on the posterior.
Unlike the Jeffreys prior, it can be used in high dimensional settings. 
The I-prior methodology can be used as a principled alternative to Tikhonov regularization, which suffers from well-known theoretical problems which are briefly reviewed. 

The regression function is assumed to lie in a reproducing kernel Hilbert space (RKHS) over a low or high dimensional covariate space, giving a high degree of generality. 
Analysis of some real data sets and a small-scale simulation study show competitive performance of the I-prior methodology, which is implemented in the R-package \verb|iprior|.
		\end{abstract}

		
		
		
		
	\end{frontmatter}
	
	

\section{Introduction}\label{sec-intro}

Consider a sample $(x_1,y_1),\ldots,(x_,y_n)$, where $y_i$ is a real-valued measurement on unit $i$, and $x_i$ lies in a set $\cX$ and represents some characteristic or collection of characteristics, numerical or otherwise, of unit $i$.
Furthermore, let $\Psi=(\psi_{ij})$ be an $n\times n$ positive definite matrix, $h$ a symmetric positive definite kernel over $\cX$, and $\cF$ a set of real-valued functions over $\cX$. 
In this paper we consider the regression model 
\begin{eqnarray}\label{regr}
y_i = f(x_i) + \varepsilon_i  \hspace{6mm}\mbox{$f\in\cF$, $x_i\in\cX$, $i=1,\ldots,n$},
\end{eqnarray}
where 
\begin{eqnarray}\label{err}
(\varepsilon_1,\ldots,\varepsilon_n) \sim \MVN(0,\Psi^{-1}).
\end{eqnarray}
Here, $\Psi$ is taken to be known up to a low dimensional parameter, e.g., $\Psi=\psi I_n$ ($\psi>0$, $I_n$ the $n\times n$ identity matrix), reflecting i.i.d.\  errors.
We shall further assume that $\cF$ is a reproducing kernel Hilbert space, i.e., $\cF$ possesses a reproducing kernel $h$ (see Section~\ref{sec-rkhs} for more details). For further reference, we write
\begin{eqnarray}\label{rk}
\mbox{$\cF$ is an RKHS over $\cX$ with reproducing kernel $h$}.
\end{eqnarray}

An RKHS is a Hilbert space of functions for which point evaluation is a continuous linear functional, i.e., functions which are sufficiently close in norm are also pointwise close. It follows that a normwise consistent estimator of $f$ is also pointwise consistent. 
The assumption that $\cF$ is an RKHS has the further benefit that the Fisher information on $f$ in~(\ref{regr}) subject to~(\ref{err}) exists.
Our proposed methodology is general in that essentially arbitrary covariate spaces $\cX$ and RKHSs $\cF$ can be used.

If the dimension of $\cF$ is high compared to $n$, the maximum likelihood (ML) estimator of $f$ is typically of little use, for example, it may interpolate the data. 
Only one generally applicable and `automatic' (i.e., requiring no additional user choices) estimation method for $f$ in~(\ref{regr}) subject to~(\ref{err}) and~(\ref{rk}) appears to have been described in the literature, namely {\em Tikhonov regularization}. The Tikhonov regularizer can be defined as the minimizer of the function from $\cF$ to $\mR$ defined by the mapping
\begin{eqnarray}\label{reg}
f \mapsto \sum_{i=1}^n\sum_{j=1}^n \psi_{ij}\big(y_i-f(x_i)\big)\big(y_j-f(x_j)\big) + \lambda^{-1}\norm{f-f_0}_{\cF}^2,
\end{eqnarray}
where $\lambda>0$ is a scale (or smoothness) parameter (usually estimated using cross-validation), $f_0$ is a prior `best guess' of $f$, and the first term on the right hand side is minus two times the log-likelihood of $f$ up to a constant.
However, \citet{cp19} showed that the Tikhonov regularizer is essentially inadmissible with respect to squared error loss if $\cF$ is infinite dimensional.

To overcome this problem, it seems reasonable to use a Bayes or empirical Bayes approach, assigning a prior to $f$ whose support is a subset of $\cF$: Wald's complete class theorem then ensures admissibility. In this paper we propose a novel objective prior for $f$, where by `objective' we mean automatically generated without further user input after the model has been chosen. 
Our prior is called {\em I-prior}, where the `I' refers to (Fisher) information, and is defined as a maximizer of entropy subject to a suitable constraint based on the Fisher information. Following \citet{jaynes03} the I-prior can thus be thought of as `least informative'.
For the present problem the I-prior is Gaussian, with prior mean $f_0$, and covariance kernel proportional to the Fisher information on $f$. Under the I-prior, $f$ has the simple representation
\begin{eqnarray}  \label{repr}
f(x) = f_0(x) + \lambda\sum_{i=1}^n h(x,x_i)w_i, \hspace{10mm}(w_1,\ldots,w_n)\sim\MVN(0,\Psi), \lambda>0.
\end{eqnarray}
Since $h(\cdot,x_i)\in\cF$ (see Section~\ref{sec-rkhs}), this representation immediately shows that I-prior realizations are in $\cF$.

From an intuitive perspective, the I-prior is reasonable because if the Fisher information on a linear functional of $f$ is high, the linear functional will have a high prior variance, and the posterior mean may be largely determined by the data; if on the other hand little Fisher information is available for a particular linear functional, the prior variance will be small, and the posterior mean may be largely determined by the prior mean.




It can be seen that the I-prior depends on the data $x_1,\ldots,x_1$ and its support is only a subset of $\cF$, which is justified in Sections~\ref{sec-inducedrkhs} and~\ref{sec-datadep}.

An alternative approach to estimating $f$ in~(\ref{regr}) is for the user to choose a prior over the space of functions $\cF$ and compute the posterior distribution. Such a choice can be made based on subjective beliefs, prior knowledge, or theoretical or practical considerations. If the chosen prior is Gaussian, this method is called Gaussian process regression (GPR). A general class of Gaussian and L{\'e}vy process priors over an RKHS was characterized by \citet{pillai07}. 

In the I-prior approach, a number of hyperparameters remain undetermined: the scale parameter $\lambda$ in~(\ref{repr}), any parameters of the error precision matrix $\Psi$, and potentially parameters of the kernel $h$. In the application Section~\ref{sec-applic}, we used an empirical Bayes approach whereby these hyperparameters are estimated by their maximum likelihood estimators. Although the philosophical aspects of empirical Bayes have not been fully resolved \citep[see][with discussion]{efron19}, estimating hyperparameters from the data is in line with the usual approach in the random effects literature and the regularization literature. Naturally, a fully Bayes approach could be used as well by assigning priors to the hyperparameters, but we did not pursue this avenue here as we expect no major differences in outcomes for the examples considered. Furthermore, it is not clear to us what could be a reasonable hyperprior for the scale parameter $\lambda$.

Though the I-prior is new, the idea of using the Fisher information to define an objective prior is not. The Jeffreys prior and Zellner's $g$-prior \citep{zellner86} are also based on the Fisher information, but in a different way. In particular, for a multiple regression model, the $g$-prior covariance matrix for the vector of regression coefficients is proportional to its {\em inverse} Fisher information matrix, in contrast to the actual Fisher information matrix for the I-prior. However, we show in Section~\ref{sec-gprior} that the {\em standard} $g$-prior can be interpreted as an I-prior, if the covariate space is equipped with the Mahalanobis distance.

\citet{jamil18} builds on an earlier version of the present paper and another unpublished manuscript. He provides a number of extensions to the present methodology, including probit and logit models using a fully Bayes approach, Bayesian variable selection using I-priors, and Nystr\"om approximations for speeding up the I-prior methodology. Furthermore, he contributed a user friendly R package \texttt{iprior} \citep{jamil19package}, further described in \citet{jb19}.


An overview of the paper is as follows. 
In Section~\ref{sec-noninf}, we give an expression for the Fisher information on the regression function and define the I-prior, illustrating with basic examples of regression with linear regression functions and one-dimensional smoothing.
In Section~\ref{sec-other}, we compare the methodology with existing methods, including Zellner's $g$-priors (which can be viewed as a special case of I-priors), cubic spline smoothing and more generally Tikhonov regularization, and Jeffreys priors. 
In Section~\ref{sec-marg}, the posterior distribution of the regression function under the I-prior is given. 
Section~\ref{sec-fbm} introduces the RKHSs used in the real data analyses of Section~\ref{sec-applic}. In particular the {\em canonical RKHS} of linear functions is briefly described and a more detailed description is given of the family of fractional Brownian motion (FBM) RKHSs over a Hilbert space.
Smoothness properties of the functions in the FBM RKHS are given, as well as of corresponding I-prior paths. 
In Section~\ref{sec-applic}, we apply the I-prior methodology to a number of data sets and compare predictive performance with a number of published results for the same data sets, showing the I-prior methodology compares well.
In Section~\ref{sec-sim}, a simulation study is done for one-dimensional smoothing, in order to compare with Tikhonov regularization and GPR with a squared exponential prior. 
The concluding Section~\ref{sec-conc} briefly summarizes the paper and gives some directions for future work. 
\ref{app-fish} puts the proposed methodology in a broader setting and may be of interest in its own right. 

The main new contribution of this paper is Section~\ref{sec-noninf}. The results in Section~\ref{sec-fbm} are largely well-known among experts, but may be difficult to find in the literature. Section~\ref{sec-estdetails} gives some numerical insights that may also be useful for Gaussian process regression. The developments in \ref{app-fish} are new except when indicated otherwise.

\section{I-priors} \label{sec-noninf}

A definition of RKHSs is recalled in Section~\ref{sec-rkhs}. The Fisher information on the regression function is derived in Section~\ref{sec-fish}. Being positive definite, the Fisher information induces a new RKHS, which is described in Section~\ref{sec-inducedrkhs}. The I-prior is defined in Section~\ref{sec-iprior}. 
We give a justification of the data dependence of the I-prior in Section~\ref{sec-datadep}, and some relatively straightforward applications, to regression with linear functions and one-dimensional smoothing, are given in Section~\ref{sec-basic}.
Except for the first subsection, all developments here are new unless indicated otherwise.

\subsection{Definition of RKHSs and of tensor products of RKHSs}\label{sec-rkhs}

Recall that a Hilbert space is a complete inner product space with a positive definite inner product.
Suppose $\cF$ is a Hilbert space of functions over a set $\cX$ equipped with the inner product $\langle\cdot,\cdot\rangle_\cF$. A symmetric function $h:\cX\times\cX\rightarrow\mR$ is a {\em reproducing kernel} of $\cF$ if and only if
\begin{itemize}
	\item[(a)] $h(x,\cdot)\in\cF$ for all $x\in\cX$
	\item[(b)] $f(x)=\langle f,h(x,\cdot)\rangle_\cF$ for all $f\in\cF$ and $x\in\cX$.
\end{itemize}
A Hilbert space of functions is called a {\em reproducing kernel Hilbert space} (RKHS) if it possesses a reproducing kernel.
If $\cX$ is a set, a function $h:\cX\times\cX\rightarrow\mR$ is said to be positive definite on $\cX$ if
$\sum_{i=1}^n\sum_{j=1}^n\alpha_i\alpha_jh(x_i,x_j)\ge 0$ for all $n=1,2,\ldots$, $\alpha_1,\ldots,\alpha_n\in\mR$, and $x_1,\ldots,x_n\in\cX$.
(Note that standard terminology is slightly different for kernels than for matrices and a positive definite kernel is, in fact, the generalization of a positive semi-definite matrix; see \citet{sfl11} for restricted versions of positive definiteness for kernels.)
By (a) and (b) above, a reproducing kernel $h$ satisfies $h(x,x')=\langle h(x,\cdot),h(x',\cdot)\rangle_\cF$, and is hence positive definite.
The Moore-Aronszajn theorem states that every symmetric positive definite kernel defines a unique RKHS.

Let $\cF_1$ and $\cF_2$ by two RKHSs over $\cX_1$ resp.\ $\cX_2$. For $f_1\in\cF_1$ and $f_2\in\cF_2$, the {\em tensor product} $f_{12}=f_1\otimes f_2$ is defined by $f_{12}(x_1,x_2)=f_1(x_1)f_2(x_2)$. 
The tensor product of $\cF_1$ and $\cF_2$ is denoted as $\cF_1\otimes\cF_2$ and is defined as the closure of the set of functions $\{f_1\otimes f_2|f_1\in\cF_1, f_2\in\cF_2\}$ equipped with the inner product
\[  \langle f_{1}\otimes f_2,f_{1}'\otimes f_2' \rangle_{\cF_1\otimes\cF_2} = \langle f_1,f_1'\rangle_{\cF_1}\langle f_2,f_2'\rangle_{\cF_2}. \]

A Hilbert space $\cB$ over $\cX$ is called a {\em feature space} of $\cF$ with feature $\phi:\cX\rightarrow\cB$ if $f(x)=\langle\phi(x),f\rangle_\cB$ for all $x\in\cX$. $\cF$ is called the {\em canonical feature space}, and has feature $h(x,\cdot)$.

A function $e_x:\cF\rightarrow\mR$ is called a {\em point evaluator} at $x$ if $e_x(f)=f(x)$. It can be shown that a Hilbert space of functions is an RKHS if and only if the point evaluators are continuous.


\begin{example}\label{ex-lin}
Let $\cF$ be the RKHS over $\cX=\mR^p$ with reproducing kernel $h(x,x')=x^\top x'$. Then $\cF$ consists of functions of the form $f(x)=x^\top\beta$ with norm $\norm{f}_\cF=\norm{\beta}_{\mR^p}$. $\cF$ is also called the {\em dual space} of $\mR^p$. 
\end{example}

\begin{example}\label{ex-bm}
Let $\cF$ be the RKHS over $\mR$ with reproducing kernel
\begin{align}\label{cbm}
h(x,x') = -\frac1{2n^2}\sum_{i=1}^n\sum_{j=1}^n\left(\abs{x-x'} - |x-x_i| - |x'-x_j| + |x_i-x_j|\right), 
\end{align}
for real numbers $x_1,\ldots,x_n$. 
Then $\cF$ is called a centered Brownian motion RKHS (see Section~\ref{sec-fbm1} for a generalization).
It follows from \citet[Section 10]{vz08rep} that $\cF$ consists of functions $f:\mR\rightarrow\mR$ possessing a square integrable derivative, satisfying $\sum f(x_i)=0$, and with norm
\[ \norm{f}_{\cF}^2 = \int \dot f(x)^2dx,  \]
where $\dot f$ denotes the derivative of $f$.
\end{example}

\subsection{The Fisher information on the regression function}\label{sec-fish}

The log-likelihood of parameter $f$ in~(\ref{regr}) subject to~(\ref{err}) is given by
\[  L(f|y) = C - \frac12\sum_{i=1}^n\sum_{j=1}^n\psi_{ij}(y_i-f(x_i))(y_j-f(x_j)) \]
for a constant $C$. 
The next lemma gives the Fisher information $I[f] = -E\nabla^2 L(f|y)$ for $f$.
\begin{lemma}\label{lem-fish}
	Suppose~(\ref{regr}) subject to~(\ref{err}) and~(\ref{rk}) holds.
	Then the Fisher information $I[f]\in\cF\otimes\cF$ for $f$ is given by
	\[  I[f] = \sum_{i=1}^n\sum_{j=1}^n \psi_{ij}h(\cdot,x_i)\otimes h(\cdot,x_j). \]
	More generally, if $\cF$ has feature space $\cB$ with feature $\phi:\cX\rightarrow\cB$, then if $f(x)=\langle\phi(x),\beta\rangle_\cB$ the Fisher information $I[\beta]\in\cB\otimes\cB$ for $\beta$ is
	\[  I[\beta] = \sum_{i=1}^n\sum_{j=1}^n\psi_{ij}\phi(x_i)\otimes\phi(x_j). \]
	For any fixed $g\in\cF$, the Fisher information on $f_g=\langle f,g\rangle_\cF$ is
	\begin{align*}
	I[f_g]  = \sum_{i,j=1}^n \psi_{ij}g(x_i)g(x_j).
	\end{align*}
\end{lemma}

\begin{remark}
	If $\cF$ is a Hilbert space of functions but not an RKHS, then there is an $x\in\cX$ such that the point evaluator $e_x(f):=f(x)$ is discontinuous \citep{aronszajn50}. Thus, if there is an $x_i$ in the sample such that the point evaluator at $x_i$ is discontinuous, the Fisher information on $f$ does not exist because the gradient of the likelihood does not exist.
\end{remark}

\noindent\proof[Proof of Lemma~\ref{lem-fish}]
For $x\in\cX$, let $e_x:\cB\rightarrow\mR$ be defined by $e_x(\beta)=\langle\phi(x),\beta\rangle_\cB$. Clearly, $e_x$ is linear and continuous. Hence, the directional derivative of $e_x(\beta)$ in the direction $\gamma\in\cB$ is
\[  \nabla_\gamma e_x(\beta) = \lim_{\delta\rightarrow 0}\frac{e_x(\beta+\delta \gamma)-e_x(\beta)}{\delta} = e_x(\gamma) = \langle \phi(x),\gamma\rangle_\cB . \]
Hence by definition of the gradient (see~\ref{app-grad})  
\begin{eqnarray}\label{gradf}
\nabla e_x(\beta) = \phi(x).
\end{eqnarray}
The log-likelihood of $\beta$ is given by
\begin{eqnarray*}
	L(\beta|\by,\Psi)
	=
	C - \frac12 \sum_{i=1}^n\sum_{j=1}^n\psi_{ij}(y_i-e_{x_i}(\beta))(y_j-e_{x_j}(\beta)),
\end{eqnarray*}
for some constant $C$.
Then after standard calculations and using~(\ref{gradf}),
\begin{eqnarray*}
	I[\beta] = -E\left[\nabla^2 L(\beta|\by,\Psi)\right] =  \sum_{i=1}^n\sum_{j=1}^n\psi_{ij}\,\nabla e_{x_i}(\beta)\otimes\nabla e_{x_j}(\beta) = \sum_{i=1}^n\sum_{j=1}^n \psi_{ij}\,\phi(x_i)\otimes\phi(x_j).
\end{eqnarray*}
Taking the canonical feature $\phi(x)=h(x,\cdot)$, the formula for $I[f]$ follows.

For any fixed $g\in\cF$, the Fisher information on $\langle f,g\rangle_\cF$ is
\begin{align*}
\lefteqn{
	I[\langle f,g\rangle_\cF] =
	\langle I[f],g\otimes g\rangle_{\cF\otimes\cF}
	= \sum_{i,j=1}^n \psi_{ij}\langle h(\cdot,x_i)\otimes h(\cdot,x_j),g\otimes g\rangle_{\cF\otimes\cF}
}\\
&
= \sum_{i,j=1}^n \psi_{ij}\langle h(\cdot,x_i),g\rangle_\cF\langle h(\cdot,x_j),g\rangle_\cF = \sum_{i,j=1}^n \psi_{ij}g(x_i)g(x_j).
\end{align*}

\endproof

\subsection{The RKHS induced by the Fisher information} \label{sec-inducedrkhs}

The Fisher information, being positive definite, induces a new RKHS over a subspace of~$\cF$. This RKHS is important because it describes the available information on $f$ in the sense explained below. We describe this RKHS next. 

Define
\begin{eqnarray}\label{fndef}
\cF_n = \Big\{ f:\cX\rightarrow\mR \Big| f(x) = \sum_{i=1}^n h(x,x_i)w_i \mbox{ for some $w_1,\ldots,w_n\in\mR$} \Big\}
\end{eqnarray}
and let $h_n$ be the kernel over $\cX$ defined by $h_n(x,x')=I[f](x,x')$, i.e.,
\begin{align}\label{hndef}
h_n(x,x') = \sum_{i=1}^n\sum_{j=1}^n \psi_{ij}h(x,x_i)h(x',x_j).
\end{align}
Note that, since $h(\cdot,x_i)\in\cF$, $\cF_n$ is a subspace of~$\cF$.
The next lemma describes the RKHS induced by the Fisher information.
\begin{lemma}\label{lem-f2}
	Let $\cF_n$ be equipped with the inner product
	\begin{equation}\label{fninprod}  \langle f_w,f_{w'}\rangle_{\cF_n}^2 = w^\top\Psi^{-1}w' \end{equation}
	where $w=(w_1,\ldots,w_n)^\top$ and $f_w(x)=\sum h(x,x_i)w_i$.
	Then $h_n$ defined by~(\ref{hndef}) is a reproducing kernel of $\cF_n$.
\end{lemma}
\noindent\proof[Proof]
Denote by $\psi_{ij}^-$ the $(i,j)$th element of $\Psi^{-1}$. By~(\ref{fninprod}) we have $\langle h(\cdot,x_i), h(\cdot,x_j) \rangle_{\cF_n}=\psi_{ij}^-$ so
\begin{align*}
\lefteqn{
	\langle f_w, h_n(x,\cdot) \rangle_{\cF_n}
	= \bigg\langle \sum_{i=1}^n h(\cdot,x_i)w_i, \sum_{j=1}^n\sum_{k=1}^n \psi_{jk} h(x,x_j)h(\cdot,x_k) \bigg\rangle_{\cF_n} }\\
&= \sum_{i=1}^n w_i \sum_{j=1}^n\sum_{k=1}^n h(x,x_j)\psi_{jk}\big\langle h(\cdot,x_i), h(\cdot,x_k) \big\rangle_{\cF_n} 
= \sum_{i=1}^n w_i \sum_{j=1}^n\sum_{k=1}^n h(x,x_j)\psi_{jk}\psi_{ik}^- \\
&= \sum_{i=1}^n w_i \sum_{j=1}^n h(x,x_j)\delta_{ij}
= \sum_{i=1}^n w_i h(x,x_i) = f_w(x).
\end{align*}
Hence, $h_n$ is a reproducing kernel for $\cF_n$.
\endproof


We immediately obtain the following interpretation of $\norm{\cdot}_{\cF_n}^2$. Let $\hat f$ be an unbiased estimator of $f$ in model~(\ref{regr}) subject to~(\ref{err}) and~(\ref{rk}). Then for $g\in\cF_n$, $\norm{g}_{\cF_n}^2$ is the Cram\'er-Rao lower bound for the variance of $\var\big(\langle g,\hat f\rangle_\cF\big)$, i.e.,
\begin{align}\label{cr}
   \var\big(\langle g,\hat f\rangle_\cF\big) \ge \norm{g}_{\cF_n}^2. 
\end{align}
It follows from the theory of weighted least squares that the lower bound is achieved if $\hat f$ is a maximum likelihood estimator of $f$. 


The next lemma implies that the data do not contain any Fisher information to distinguish between two functions $f$ and $f'$ if $f(x_i)=f'(x_i)$ for $i=1,\ldots,n$.  
\begin{lemma}\label{lem-orth}
	The orthogonal complement of $\cF_n$ in $\cF$ is 	
	\begin{eqnarray}\label{rndef}
	\cF_n^\perp = \Big\{ f\in\cF \Big| f(x_1) = \ldots = f(x_n) = 0 \Big\}.
	\end{eqnarray}
	We can hence uniquely decompose $f\in\cF$ as
	\begin{eqnarray}\label{fnip1}
	f(x) = f_n(x) + r_n(x) \quad f_n\in\cF_n,r_n\in\cF_n^\perp.
	\end{eqnarray}
	Then $I[f_n]=I[f]$ and $I[r_n]=0$. Furthermore, the Fisher information on any nonzero linear functional of $f_n$ is strictly positive.
\end{lemma}
\begin{proof}[Proof of Lemma~\ref{lem-orth}]
	Let $f_n=\sum_{i=1}^n h(\cdot,x_i)w_i\in\cF_n$ and let $g\in\cF$.
	Then by the reproducing property of $h$, $\langle f_n,g\rangle_\cF=\sum_{i=1}^n w_i\langle h(x_i,\cdot),g\rangle_\cF=\sum_{i=1}^n w_ig(x_i)$. But this vanishes for any $w_1,\ldots,w_n$ iff $g(x_1)=\ldots g(x_n)=0$, proving~(\ref{rndef}).

	Denote the log-likelihood of a parameter by $L(\cdot|y)$. Since $L(f|y)=L(f_n|y)$ and $f_n\perp_\cF r_n$, the definition of Fisher information immediately implies $I[f_n]=I[f]$ and $I[r_n]=0$.
	For $g\in\cF$, $\langle f_n,g\rangle_\cF\ne 0$ iff $g\not\in\cF_n^\perp$. But then, if $g\in\cF\setminus\cF_n^\perp$, Lemma~\ref{lem-fish}  implies that  $I[\langle f_n,g\rangle_\cF]=\sum\psi_{ij}g(x_i)g(x_j)>0$. 
\end{proof}

\begin{remark}
	Another way to obtain the conclusion that the data contain no Fisher information to distinguish between two functions which have the same values at $x_1,\ldots,x_n$ is as follows. 
	The {\em Fisher information metric} over $\cF$ is the distance induced by semi-norm over $\cF$ given by
	\[  \norm{f}_I^2 = \big\langle I[f], f\otimes f \big\rangle_{\cF\otimes\cF} = \sum_{i=1}^n\sum_{j=1}^n\psi_{ij}f(x_i)f(x_j). \]
	The quantity $\norm{f-f'}_I$ can be thought of as the `amount of information between $f$ and $f'$'.
	We see that $\norm{f-f'}_I=0$ if and only if $f(x_i)=f'(x_i)$ for $i=1,\ldots,n$. 
\end{remark}

\begin{example} 
	Suppose $f(x)=\sum_{i=1}^n h(x,x_i)w_i\in\cF_n$ and the errors in~(\ref{regr}) are autoregressive, in particular, $\varepsilon_1=\eta_1$, $\varepsilon_{i+1}=\alpha\varepsilon_{i}+\eta_i$ ($i=2,\ldots,n$), with the $\eta_i$ i.i.d.\  $N(0,\sigma^2)$ and $-1\le\alpha\le 1$.
	Then from Lemma~\ref{lem-f2} and \ref{app-ar},
	\begin{eqnarray*}
		\norm{f}_{\cF_n}^2 = w^\top\Psi^{-1}w = \frac{1}{\sigma^2}\sum_{i=1}^n\Big(\sum_{j=i}^n\alpha^{j-i}w_j\Big)^2,
	\end{eqnarray*}
	where $0^0:=1$.
	We have the special cases
	\begin{eqnarray*}
		\norm{f}_{\cF_n}^2 = \frac{1}{\sigma^2}\times
		\left\{
		\begin{array}{ll}
			\sum_{i=1}^{i-1} (w_{i+1}-w_i)^2 & \alpha=-1\\
			\sum_{i=1}^n w_i^2 & \alpha=0\\
			\sum_{i=1}^n\big(\sum_{j=i}^nw_j\big)^2 & \alpha=1
		\end{array}
		\right. .
	\end{eqnarray*}
\end{example}

\begin{example}\label{ex-smooth}
	(Continuation of Example~\ref{ex-bm}.)
	With $\cF$ the centered Brownian motion RKHS over $\mR$, it can be seen that $\cF_n$ is the set of functions which integrate to zero and are piecewise linear with knots at $x_1,\ldots,x_n$.
	Hence for $f\in\cF_n$,
	\[ \norm{f}_{\cF}^2 = \int \dot f(x)^2dx = \sum_{i=1}^{n-1}\frac{(f(x_{i+1})-f(x_i))^2}{x_{i+1}-x_i} . \]
	Furthermore, if $f\in\cF_n$ satisfies $f(x)=\sum h(x,x_i)w_i$ and assuming $x_1\le x_2,\ldots\le x_n$, then it is straightforward to check (by substituting $f(x_k)=\sum h(x_k,x_i)w_i$ into the right hand sides) that
	\begin{eqnarray*}
		w_1 = \frac{f(x_{2})-f(x_{1})}{x_{2}-x_{1}}, \quad\quad w_n = \frac{f(x_{n})-f(x_{n-1})}{x_{n}-x_{n-1}}
	\end{eqnarray*}
	and for $i=2,\ldots,n-1$,
	\begin{eqnarray*}
		w_i = \frac{f(x_{i+1})-f(x_{i})}{x_{i+1}-x_{i}} - \frac{f(x_{i})-f(x_{i-1})}{x_{i}-x_{i-1}}.
	\end{eqnarray*}
	It follows that $f\in\cF_n$ can be represented as
	\begin{align}\label{fint}
	f(x) = \int_{-\infty}^x \beta(t)dt ,
	\end{align}
	where
	\begin{align}\label{betadif}
	\beta(t) = \sum_{i:x_i\le t}w_i = \frac{f(x_{i_t+1})-f(x_{i_t})}{x_{i_t+1}-x_{i_t}}, 
	\end{align}
	with $i_t=\max_{x_i\le t}i$. Note that $\sum w_i=0$ and hence $\lim_{t\rightarrow\pm\infty}\beta(t)=0$.
	
	By Lemma~\ref{lem-f2}, $\norm{f}_{\cF_n}^2=w^\top\Psi^{-1}w$. For i.i.d.\  errors, the above expressions for the $w_i$ show this is proportional to
	\begin{align}\label{pen} \Big(\frac{f(x_{2})-f(x_{1})}{x_{2}-x_{1}}\Big)^2 + 
	\sum_{i=2}^{n-1} \Big(\frac{f(x_{i+1})-f(x_{i})}{x_{i+1}-x_{i}} - \frac{f(x_{i})-f(x_{i-1})}{x_{i}-x_{i-1}}\Big)^2
	+ \Big(\frac{f(x_{n})-f(x_{n-1})}{x_{n}-x_{n-1}}\Big)^2. 
	\end{align}
\end{example}

\subsection{Definition of I-priors}\label{sec-iprior}

By Lemma~\ref{lem-orth}, the set $\cF$ is too big for the purpose of estimating $f$, in the sense that, for pairs of functions in $\cF$ with the same values at $x_1,\ldots,x_n$, the data do not contain information on whether one is closer to the truth than the other.
An objective prior for $f$ therefore need not have support $\cF$, instead it is sufficient to consider priors with support $f_0+\cF_n$, where $f_0\in\cF$ is fixed and chosen a priori as a `best guess' of $f$. Lemma~\ref{lem-orth} implies the data contain information to allow a comparison between any pair of functions in $f_0+\cF_n$.

We follow \citet{jaynes57a,jaynes57b,jaynes03} and define an objective prior using the maximum entropy principle. 
The entropy of a prior $\pi$ over $f_0+\cF_n$ relative to a measure $\nu$ is defined as
\[  {\cal E}(\pi) = - \int_{f_0+\cF_n}\pi(f)\log\pi(f)\nu(\diff f). \]
We take $\nu$ to be volume measure induced by $\norm{\cdot-f_0}_{\cF_n}$, which is flat. 
An I-prior for $f$ is now defined as a prior maximizing entropy subject to a constraint of the form
\[  E_\pi\norm{f-f_0}_{\cF_n}^2= \mbox{constant}. \]
Variational calculus shows that I-priors for $f$ are the Gaussian variables with mean $f_0$ and covariance kernel proportional to $h_n$ given by~(\ref{hndef}), i.e., 
\[  \cov_{\pi}(f(x),f(x')) = \lambda \sum_{i=1}^n\sum_{j=1}^n \psi_{ij}h(x,x_i)h(x',x_j), \]
for some $\lambda>0$. 
Thus, if $f$ has an I-prior distribution, we can use the convenient representation~(\ref{repr}).

The posterior distribution of $f$ and the marginal likelihood of $(\lambda,\Psi)$ are given in Section~\ref{sec-marg}.

\subsection{Data dependence of I-prior}\label{sec-datadep}

It can be seen that the I-prior depends on the data $x_1,\ldots,x_1$. An argument can be made that any objective prior must in fact be data dependent, and representable in the form 
\begin{align}\label{ranrep}
f(x) = f_0(x) + \sum_{i=1}^n h(x,x_i)\alpha_i \mbox{ for random $\alpha_i$}.
\end{align}
The argument is as follows (details on the assertions are given in Section~\ref{sec-fish}). Any $f\in\cF$ can be uniquely decomposed as $f(x)=f_n(x)+r_n(x)$, where $f_n(x)=\sum_{i=1}^n h(x,x_i)w_i$ for some $w_1,\ldots,w_n$ and $r_n(x_i)=0$ for $i=1,\ldots,n$. Since the likelihood for $f$ does not depend on $r_n$, and $f_n$ and $r_n$ are orthogonal in $\cF$, the data contain no (Fisher) information on $r_n$. Therefore, unless we have actual prior information about the relation between $r_n$ and $f_n$, it is not possible to do statistical inference on $r_n$ using the data at hand. If the prior for $f$ is representable as in~(\ref{ranrep}), this implies the  prior for $r_n$ is a point mass at our prior guess of it, and so is the posterior for $r_n$ (note that our prior guess for $r_n$ is the orthogonal projection of $f_0$ onto the subspace of $\cF$ consisting of functions $r$ for which $r(x_1)=\ldots=r(x_n)=0$). 
In summary, only with a prior representable as in~(\ref{ranrep}), all the `information' available about $r_n$ is our prior guess of it, and it remains nothing more than a mere prior guess even after observing the data. Our maximum entropy argument in Section~\ref{sec-iprior} and \ref{app-fish} then leads to the I-prior represented in~(\ref{repr}).

\subsection{Applications}\label{sec-basic}

We give the I-priors for linear regression functions and for one-dimensional smoothing, summarized in Table~\ref{tbl-basic}

\begin{table}
	\begin{tabular}{llll}
		$\cX$                           & kernel $h(x,x')$   & $f(x)$                         & I-prior \\ \hline
		$\mR^p$ (dot product)           & $x^\top x'$        & $x^\top\beta$                  & $\beta\sim\MVN(\beta_0,\lambda X^\top X)$    \\
		$\mR^p$ (Mahal.\ metric)        & $x^\top(X^\top X)^{-1} x'$        & $x^\top\beta$                  & $\beta\sim\MVN(\beta_0,g(X^\top X)^{-1})$    \\
		$\mR$                           & centered BM        & $\int_{-\infty}^x\beta(t)dt$   & $\beta\sim$ Brownian bridge    \\
		$\mR^p$                         & centered BM        & $f$ is H\"older $\ge 1/2$& $f$ is H\"older 1 (a.s.)    \\
	\end{tabular}
	\caption{I-priors for the illustrative examples in Section~\ref{sec-basic} (first three rows), assuming model~(\ref{regr}) with i.i.d.\  normal errors. The model in the last row is discussed in Section~\ref{sec-fbm}. The I-prior when $\mR^p$ is equipped with the Mahalanobis metric is also known as the $g$-prior. BM stands for Brownian motion (the same as FBM-1/2). See the text for further details.}
	\label{tbl-basic}
\end{table}

\subsubsection{Linear regression functions}\label{sec-lin}

Consider the model
\[  y_i = x_i^\top\beta + \varepsilon_i, \hspace{10mm}x_i\in\mR^p,\, i=1,\ldots,n \]
subject to~(\ref{err}). With $X$ the $n\times p$ matrix whose $i$th row is $x_i^\top$, the Fisher information on $\beta$ is
\[  I[\beta] = X^\top\Psi X . \]
Hence, the I-prior for $\beta$ with prior mean $0$ is the multivariate normal distribution with covariance matrix $\lambda X^\top\Psi X$,
i.e., under the I-prior
\begin{equation}\label{liniprior}  \beta\sim\MVN(\beta_0,\lambda X^\top\Psi X) \end{equation}
for a scale parameter $\lambda>0$ and a prior mean $\beta_0\in\mR^p$. 
This prior is suitable whether $p$ is small or large (see Section~\ref{sec-applic} where the prior is used for potentially large $p$).

In the above, we assumed $\cF$ is the dual space of $\mR^p$, i.e., if $f(x)=x^\top\beta$,
$\norm{f}_\cF = \norm{\beta}_{\mR^p}$.
This is a suitable space if, say, a vector $x$ is a series of repeated measurements on the same scale, but is not suitable if $x$ consists of measurements on different scales, such as height in metres and weight in kilograms. In that case, it is better to instead equip $\mR^p$ with the Mahalanobis distance, i.e.,
\[  \norm{\beta}_\text{Mah}^2=\beta^\top(X^\top\Psi X)^{-1}\beta . \]
With this metric, the Fisher information on $\beta$ is $(X^\top\Psi X)^{-1}$ (rather than $X^\top\Psi X$ in the standard Euclidean metric), and the I-prior becomes
\begin{align}\label{gprior}  \beta\sim\MVN(\beta_0,\lambda\,(X^\top\Psi X)^{-1}).   \end{align}
This is the usual $g$-prior with $g=\lambda$ \citep{zellner86}.

In contrast to~(\ref{liniprior}), the prior~(\ref{gprior}) is scale invariant and hence suitable for covariates measured on different scales. However, it has the drawback that it is only suitable if $p\ll n$, whereas~(\ref{liniprior}) can be used even if $p>n$.
In \ref{app-fish} a different derivation of the $g$-prior is given as well as a generalization.

\subsubsection{One-dimensional smoothing with I-priors and connection with cubic spline smoothing}\label{ex-smooth2}

We continue Examples~\ref{ex-bm} and \ref{ex-smooth}, where we assumed that the regression function lies in the centered Brownian motion RKHS over $\mR$.

With i.i.d.\  errors, under the I-prior the $w_i$ in~(\ref{betadif}) are i.i.d.\  zero mean normals, 
so that $\beta$ defined there is an ordinary Brownian bridge with respect to the empirical distribution function $P_n(x)=\sum_{i=1}^nI(x_i<x)$. It is straightforward to verify that $\beta$ then has covariance kernel
\begin{align*}
\cov(\beta(x),\beta(x')) = n^2\cov(P_n(x),P_n(x')) = n\big[\min(P_n(x),P_n(x'))-P_n(x)P_n(x')\big].
\end{align*}
From~(\ref{fint}), the prior process for $f$ is thus an integrated Brownian bridge. This shows a close relation with cubic spline smoothers, which can be interpreted as the posterior mean when the prior is an integrated Brownian motion
\citep[][Section~3.8.3]{wahba78,wahba90,gs94}.
Under the I-prior, we have $\var(\beta(x))=P_n(X<x)(1-P_n(X<x))$, which shows an automatic boundary correction: close to the boundary there is little Fisher information on the derivative of the regression function, so the prior variance is small. This will then lead to more shrinkage of the posterior derivative of $f$ towards the derivative of the prior mean.

Note that the problem of finding the posterior mean of $f$ under the I-prior can be formulated as a penalized generalized least squares problem with penalty proportional to $\norm{f}_{\cF_n}^2$ which is proportional to~(\ref{pen}). 

The natural cubic spline smoother and I-prior estimator under the Brownian motion RKHS are hence similar, but have the following main differences (we assume for simplicity that the prior mean is zero). In the range of the observed $x$-values, the former is piecewise cubic and the latter is piecewise linear; outside this range, they are linear and constant, respectively.
However, the two methods are based on different models: due to the penalty $\int\ddot f(x)^2dx$, the cubic spline smoother assumes two derivatives, whereas the I-prior estimator only assumes one, i.e., at least from a theoretical perspective the I-prior has broader applicability.

In the present setting, the smoother the errors (e.g., the more positively autocorrelated the errors are), the more difficult it is to estimate the regression function. This is because smoother errors are more like a function in the RKHS than rougher errors. The I-prior accommodates for this fact by roughening the prior, so that rough functions in the RKHS can still be estimated reasonably even if the errors are relatively smooth.
Let us consider AR(1) or MA(1) errors.
If the errors are dependent, $\beta$ is a generalized Brownian bridge because, whilst being tied to zero outside the range of the $x_i$s, the increments which are summed over are dependent.
Note that $\beta$ is piecewise constant with jumps at the $x_i$, so the I-prior for $f$ is piecewise linear with knots at the $x_i$, and the same holds true for the posterior mean.
As follows from Lemma~\ref{lem-ar} in \ref{app-ar}, if the errors are an AR(1) process with parameter $\alpha$ and error variance $\sigma^2$, the $w_i$ form an MA(1) process with parameter $-\alpha$ and error variance $\sigma^{-2}$, and if the errors are an MA(1) process with parameter $\alpha$ and error variance $\sigma^2$, the $w_i$ form an AR(1) process with parameter $-\alpha$ and error variance $\sigma^{-2}$.
It follows that if the errors are a random walk (i.e., and AR(1) process with parameter $\alpha=1$) then it can be checked that the I-prior is also a random walk and the model is not identified, i.e., the I-prior has (essentially) the same distribution as the errors and the regression curve cannot be separated from the errors. Thus, if the errors form a random walk, and all we know about the regression curve is that it is weakly differentiable, there is no way of determining what part of the variation in the $y_i$s is due to the regression curve or due to the errors. To estimate $f$, a stronger assumption has to be made, e.g., that it is twice weakly differentiable.



\section{Comparison with other methods}\label{sec-other}

We now give a brief overview of some other existing methods for estimating the regression function in~(\ref{regr}).
In Section~\ref{sec-intro} we compared the I-prior methodology with Gaussian process regression. 

\subsection{Zellner's $g$-priors}\label{sec-gprior}

Zellner's $g$-prior \citep{zellner86} is a scale invariant prior for linear regression functions, and is suitable when covariates are measured on different scales (such as height in metres and weight in kilograms).
In Section~\ref{sec-lin} we showed that it can be viewed as a special case of the I-prior, when the covariate space is equipped with Mahalanobis distance. 
Alternatively, in~\ref{app-maxent} a different construction (and a generalization) of the $g$-prior is given, namely as a maximum entropy prior subject to a constraint involving the Fisher information metric; this construction is essentially the one originally given by Zellner.
A drawback of the $g$-prior is that it cannot be used if the number of covariates $p$ is large compared to the sample size $n$.

\subsection{Cubic spline smoothing}

If $\cF$ is the aforementioned centered Brownian motion RKHS over $\mR$ (see Section~\ref{ex-smooth2} for more details) and the errors are i.i.d., I-prior estimation is similar to cubic spline smoothing. Whereas the cubic spline smoother minimizes~(\ref{reg}) with $\norm{f}_\cF^2=\int\ddot f(x)^2dx$ and $\psi_{ij}=I(i=j)$ ($I$ is the indicator function), the I-prior estimator minimizes a similar expression with $\int \ddot f(x)^2dx$ replaced by~(\ref{pen}).
Note that if the $x_i$ are equally spaced, then~(\ref{pen}) multiplied by $n^2$ is a discrete approximation of $\int \ddot f(x)^2dx$.

Although the two methods will therefore tend to yield similar smoothers, there is a big theoretical difference between the two methods, in that I-prior estimation only assumes $f$ has one derivative, while cubic spline smoothing assumes two derivatives.

\subsection{Tikhonov regularization}

The Tikhonov regularizer of a regression function $f$ is the function minimizing~(\ref{reg}). It is well-known to have a Bayesian interpretation, namely as the posterior mean of $f$ when the prior is a Gaussian with mean $f_0$ and covariance kernel $\lambda h$. 
Although at first sight Tikhonov regularization seems intuitively reasonable, it has been shown to be inadmissible with respect to squared error loss \citep{cp19}. In particular, it may undersmooth every true regression function in $\cF$ in the sense we explain now. In the Bayesian interpretation of regularization, for infinite dimensional $\cF$, the prior probability of $\cF$ is well known to be zero \citep[e.g.][Section~4.1]{lifshits12}. Undersmoothing can be said to occur if the prior function paths are, with probability one, rougher than those in $\cF$. This happens, for example, if $\cF$ is a (centered) Brownian motion RKHS, in which case the prior sample paths have regularity 0.5, while the functions in the RKHS have regularity at least 1 (see Figure~\ref{fig-pathsreg} for an illustration and Section~\ref{sec-reg} for more details). Although optimal asymptotic convergence rates can often still be obtained with undersmoothing \citep{vz07}, the simulations in Section~\ref{sec-sim} show that for finite samples undersmoothing can have significant adverse effects on the estimation of functions in $\cF$. 

\begin{figure}[tbp]
	\centering  
	\includegraphics[width=80mm]{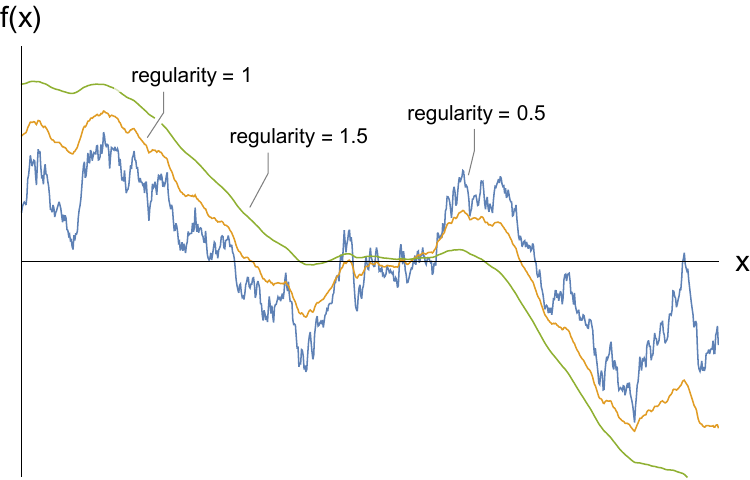}
	\caption{Randomly generated paths of different regularity. The path with regularity 0.5 is a centered Brownian motion path. Functions in the corresponding centered Brownian motion RKHS (also called FBM-1/2 RKHS) have regularity greater than 1, and can be seen to be significantly smoother than the corresponding process paths.}
	\label{fig-pathsreg}
\end{figure}


\subsection{Jeffreys priors}

Like the I-prior, the Jeffreys prior is based on the Fisher information, in particular, the Jeffreys prior is proportional to the square root of the determinant of the Fisher information. Hence, it is suitable only for low-dimensional problems. An interesting property of the Jeffreys prior is that it is invariant to parameterization, which the I-prior is not (this is easy to see as the I-prior depends on the reproducing kernel of the RKHS). For model~(\ref{repr}) subject to~(\ref{err}) and~(\ref{rk}), the Fisher information on the regression function is given by Lemma~\ref{lem-fish} and can be seen not to depend on $f$, so the Jeffreys prior is flat, and for the purposes of this paper not very useful except potentially in low-dimensional regression \citep[e.g.][]{il91}. A more extensive discussion is given in \ref{app-fish}.

\subsection{Reference priors}

Bernardo introduced {\em reference priors} \citep{bernardo79,bernardo05,bbs09}, which for one-dimensional parameters coincide with Jeffreys priors. Consider a family of probability distributions $P(x|\theta)$, $\theta\in\Theta$. A reference prior $\pi$ for $\theta$ maximizes expected Kullback-Leibler divergence of the prior from the posterior $\pi(\theta|x)$, that is, it maximizes
\[  \pi \mapsto E_{X\sim P}\big\{\text{\rm KL}(\pi(\cdot|X)\,|\,\pi)\big\} = E_{X\sim P}\int_\Theta\pi(\theta|X)\log\frac{\pi(\theta|X)}{\pi(\theta)}d\theta .  \]
Like the Jeffreys prior, reference priors are parameterization invariant, and unlike the I-prior is only suitable for low-dimensional parameters.

\subsection{Fisher kernels}	

\citet{jh98} introduced the {\em Fisher kernel}, defined for a broad range of models, which can be used with kernel methods, for example in support vector machines or as a covariance kernel in Gaussian process regression. Like I-prior, it is a method based on the Fisher information, but that is the only connection.
Suppose $P(x|\theta)$ is a probability function depending on a parameter $\theta\in\mR^p$. With $s_x(\theta)=\nabla_\theta\log P(x|\theta)$ the score vector for $\theta$ and $I[\theta]$ the Fisher information on $\theta$, the Fisher kernel is defined as
\[  K(x,x') = \langle s_x,s_{x'}\rangle_{\cF_n} = s_{x}(\theta)^\top I[\theta]^{-1}s_{x'}(\theta). \]

\section{Posterior distribution of the regression function under the \mbox{I-prior}}\label{sec-marg}

This section contains no new results, but is provided for convenience as the notation is different than for standard Gaussian process regression.  


Denote $y=(y_1,\ldots,y_n)^\top$, $\bff=(f(x_1),\ldots,f(x_n))^\top$, $\bff_0=(f_0(x_1),\ldots,f_0(x_n))^\top$, $\varepsilon=(\varepsilon_1,\ldots,\varepsilon_n)^\top$, $w=(w_1,\ldots,w_n)^\top$ and let $H$ be the $n\times n$ matrix with $(i,j)$th coordinate $h(x_i,x_j)$.
Then~(\ref{regr}) implies $y = \bff + \varepsilon$.
Under the I-prior,
\[  \bff \sim \MVN(\bff_0,\lambda^2 H\Psi H). \]
The marginal distribution of $y$ then is 
\begin{eqnarray}\label{ydist}
y \sim \MVN(\bff_0,V_{y}),
\end{eqnarray}
where the marginal covariance is given as
\[  V_{y}=\lambda^2 H\Psi H+\Psi^{-1}. \] 
Thus, the marginal log likelihood of $(\lambda,\Psi)$ is
\begin{equation}\label{marglik}  
L(\lambda,\Psi|y) =  -\frac{n}{2}\log(2\uppi) - \frac12\log|V_y| - \frac12(y-\bff_0)^\top V_y^{-1}(y-\bff_0). 
\end{equation}
The maximum likelihood (ML) estimate $(\hat\lambda,\hat\Psi)$ of $(\lambda,\Psi)$ maximizes $L(\lambda,\Psi|y)$, and its asymptotic distribution  can be found from the Fisher information. In particular, assume $\lambda=\lambda(\theta)$ and $\Psi=\Psi(\theta)$ are sufficiently smooth functions of $\theta$. Then straightforward  calculations give the well-known result that the Fisher information matrix $U$ for $\theta$ has $(i,j)$th coordinate
\[ u_{ij} = \frac12\tr\Big(V_y^{-1}\frac{\partial V_y}{\partial\theta_i}V_y^{-1}\frac{\partial V_y}{\partial\theta_j} \Big), \]
where the derivatives are applied to each coordinate of the matrix.
Now under suitable asymptotic conditions on $V_{y,\pi}$, $\sqrt{n}(\hat\theta-\theta)$ has an asymptotic multivariate normal distribution with mean zero and covariance matrix $U^{-1}$.

The next lemma gives the posterior distribution of $f$ under the I-prior.
\begin{lemma}\label{lem-post}
	The posterior distribution of $f$ in~(\ref{regr}) subject to~(\ref{err}) given $(y_1,\ldots,y_n)$ under the I-prior $\pi$ is Gaussian with mean given by
	\[  E_\pi\big[f(x)|y_1,\ldots,y_n\big] = f_0(x) + \lambda\sum_{i=1}^n h(x,x_i)\hat w_i  \]
	where	
	$\hat w    = \lambda\Psi H ^\top V_y^{-1}(y-\bff_0)$,
	and covariance kernel given by
	\begin{align*}
	\cov_{\pi}\big(f(x),f(x')|y_1,\ldots,y_n\big) &= \lambda^2\sum_{i=1}^n\sum_{j=1}^n h(x,x_i)h(x',x_j)(V_{y}^{-1})_{ij}
	\end{align*}.
\end{lemma}


\noindent\begin{proof}[Proof of Lemma~\ref{lem-post}]
	Under the I-prior $\pi$,~(\ref{repr}) holds and the joint distribution of $(w,y)$ is given by
	\[
	\binom{w}{y} \sim \MVN\left[\binom{0}{\bff_0} ,\left(\begin{array}{cc}\Psi & \lambda\Psi H ^\top \\ \lambda H\Psi & V_y\end{array}\right) \right].
	\]
	From this standard results give the posterior distribution of $w$ given $y$, i.e, the conditional distribution of $w$ given $y$, which is multivariate normal with mean $\hat w$
	and covariance matrix
	\begin{eqnarray}\label{tildew}
	\tilde V_w = \Psi - \lambda^2\Psi H ^\top V_y^{-1}H \Psi = V_y^{-1} ,
	\end{eqnarray}
	where the last equality is the Woodbury matrix identity.
	The posterior mean of $f$ is now obtained by substituting each $w_i$ in~(\ref{repr}) by $\tilde w_i$, and the posterior covariance matrix is as in the lemma.
\end{proof}

It follows from the lemma that given the I-prior, the posterior of $f$ can be represented by the left part of~(\ref{repr}) where $(w_1,\ldots,w_n)^\top$ is multivariate normal with mean $\tilde w$ and covariance matrix $V_y^{-1}$.
The computational complexity of computing the posterior distribution is $O(n^3)$, the same as in Gaussian process regression.
This can be reduced in very specific cases, such as for parametric (i.e., finite dimensional) models or for one dimensional smoothing via the Reinsch algorithm \cite[Section~2.3.3]{gs94}.
A number of approximation methods to overcome this computational problem is listed in Chapter~8 of \citet{rw06}.

\section{I-priors and FBM RKHSs}\label{sec-fbm}

The main aim of this section is to describe smoothness properties of functions in the FBM RKHS and those of I-prior paths when the regression function is assumed to be in the FBM RKHS. 
The results on smoothness of Gaussian process paths and functions in the associated RKHSs given in this section are well-known, but explicit references may be difficult to find. However, very general related results can be found in \citet{steinwart18}.

In Section~\ref{sec-fbm1}, the FBM RKHS with Hurst coefficient $\gamma$ is defined, and in Section~\ref{sec-cent} centering of an RKHS is defined.
Section~\ref{sec-hol} concerns H\"older smoothness, and the main results are that functions in the FBM RKHS with Hurst coefficient $\gamma$ are H\"older of order $\gamma$, while I-prior paths are H\"older of order $2\gamma$. This can be compared with FMB-$\gamma$ process paths, which are H\"older of any order less than $\gamma$. 
Section~\ref{sec-reg} concerns a different concept of smoothness, called regularity, which is based on the rate of decay of Karhunen-Loeve coefficients. 
This concept is perhaps most useful for one-dimensional functions. Differently from H\"older smoothness, regularity shows a {\em gap} in smoothness between FBM process paths and functions in the FBM RKHS (see also Figure~\ref{fig-pathsreg}).
The results of Sections~\ref{sec-hol} and~\ref{sec-reg} are summarized in Table~\ref{tbl-smooth}.
In Section~\ref{ex-smooth2}, we look in some detail at smoothing with a one-dimensional Brownian motion RKHS. In this case, the I-prior methodology with i.i.d.\  errors gives similar results as cubic spline smoothing, but with a different theoretical justification. 

\begin{table}
	\begin{tabular}{lll}
		Type of function & H\"older degree & Regularity (one dimensional case)  \\ \hline
		FBM-$\gamma$ process paths  & Any $<\gamma$    & $\gamma$   \\
		FBM-$\gamma$ RKHS functions & $\ge\gamma$               & $>\gamma+1/2$ \\
		FBM-$\gamma$ I-prior paths  & $2\gamma$              & $2\gamma+1$ (asymptotically if errors i.i.d.)
	\end{tabular}
	\caption{Smoothness of functions related to the FBM-$\gamma$ kernel. It is seen that all functions in an FBM RKHS are smoother than the corresponding FBM process paths, while the RKHS contains both rougher and smoother functions than I-prior paths. Note that, with probability 1, the FBM RKHS does not contain an FBM path but does contain an I-prior path.}
	\label{tbl-smooth}
\end{table}

\subsection{Canonical and Fractional Brownian motion RKHS}\label{sec-fbm1}

In this paper we consider two (families of) RKHSs of functions over a Hilbert space $\cX$ equipped with the inner product $\inn_\cX$. 

Firstly, the {\em canonical RKHS} is the dual space of $\cX$ and is defined by the canonical kernel
\[  h(x,x') = \langle x,x'\rangle_\cX. \]
Being the dual space, it consists of all linear functions over $\cX$. The Riesz representation theorem implies that for any linear function $f$ over $\cX$ there exists a $\beta\in\cX$ such that $f(x)=\langle x,\beta\rangle_\cX$. In that case, $\norm{f}_\cF=\norm{\beta}_\cX$. 

Secondly, we consider the Fractional Brownian Motion (FBM) RKHS.
\citet{schoenberg37} has shown that, for $0<\gamma<1$, there exists a Hilbert space $\cB$ and a function $\phi_\gamma:\cX\rightarrow\cB$ such that
\begin{eqnarray}\label{emb}
\norm{\phi_{\gamma}(x)-\phi_{\gamma}(x')}_\cB = \norm{x-x'}_\cX^{\gamma} \hspace{10mm}\forall x,x'\in\cX.
\end{eqnarray}
Using the polarization identity, we obtain
\begin{eqnarray}\label{fbm}
h_{\gamma}(x,x') = \langle\phi_\gamma(x),\phi_\gamma(x')\rangle_\cB = -\frac12\Big(\norm{x-x'}_\cX^{2\gamma}-\norm{x}_\cX^{2\gamma}-\norm{x'}_\cX^{2\gamma}\Big).
\end{eqnarray}
From its construction, it is clear that $h_\gamma$ is positive definite. It is in fact the covariance kernel of fractional Brownian motion (\fbm) on $\cX$ with Hurst coefficient~$\gamma$ \citep{kolmogorov40wiener,mn68}.
Note that if $\gamma=1$ then $h_\gamma(x,x')$ is the canonical kernel $\langle x,x'\rangle_\cX$.
Following \citet{cohen02}, we call the RKHS with kernel $h_\gamma$ the \fbm\ RKHS of order $\gamma$, and we denote it $\cF_\gamma$. An alternative name is the Cameron-Martin space of \fbm\ \citep[see, e.g.][]{picard11}.

\subsection{Centering of an RKHS}\label{sec-cent}

The functions in an RKHS may be arbitrarily positioned, for example, if $f$ is in a canonical or in an FBM RKHS, then $f(0)=0$, which is undesirable for the purposes of this paper. To remedy this, an RKHS may be centered. The functions in a centered RKHS have zero mean. 

If $P$ is a probability distribution over $\cX$ and $X,X'\sim P$ are independent, a kernel $h$ over $\cX$ may be centered as follows:
\[  h_\text{cent}(x,x') = E_P(h(x,x')-h(x,X)-h(x',X')+h(X,X')\big).  \]
The RKHS with kernel $h_\text{cent}$ is then centered in the sense that $E_P(f(X))=0$ for all functions $f$ in the RKHS.

In the present paper we center with respect to the empirical distribution of data $x_1,\ldots,x_n$, so that in the centered RKHS, $\sum_{i=1}^nf(x_i)=0$.
The centered canonical kernel then becomes $h_\text{cent}(x,x')=\langle x-\bar x,x-\bar x\rangle_\cX$ where $\bar x=n^{-1}\sum_{i=1}^nx_i$. The centered FBM kernel becomes
\begin{eqnarray}\label{smooth}
h_{\gamma,P}(x,x') =  -\frac1{2n^2}\sum_{i=1}^n\sum_{j=1}^n\big(\norm{x-x'}_\cX^{2\gamma}-\norm{x-x_i}_\cX^{2\gamma}-\norm{x'-x_j}_\cX^{2\gamma}+\norm{x_i-x_j}_\cX^{2\gamma}\big).
\end{eqnarray}


\subsection{H\"older smoothness}\label{sec-hol}

A function $f$ over a set $\cX$ with norm $\norm{\cdot}_\cX$ is H\"older of order $0<\gamma\le 1$ if there exists a $C>0$ such that
\begin{eqnarray}\label{holder}
\abs{f(x) - f(x')}  < C\norm{x-x'}_\cX^\gamma \quad \forall x,x'\in\cX
\end{eqnarray}
and $f$ is H\"older of order $1<\gamma\le 2$ if
\begin{align}\label{holder2}
\left|f(x+t)-2f(x)+f(x-t)\right| \le K\norm{t}_\cX^\gamma  \quad\forall x,t\in\cX,
\end{align}
for some $K>0$ (see \citeauthor{gt98}, \citeyear{gt98}, Chapter 4, or \citeauthor{stein70}, \citeyear{stein70}, Section 4.3).

It is well-known that realizations of an FBM-$\gamma$ process are a.s.\ H\"older continuous of any order less than $\gamma$ (e.g., Theorem 4.1.1 in \citeauthor{em02}, \citeyear{em02}). The next lemma shows that functions in the FBM RKHS are strictly smoother than FBM realizations.

\begin{lemma}\label{hold1}
	The functions in the FBM-$\gamma$ RKHS are H\"older of order $\gamma$. 
\end{lemma}
\begin{proof}[Proof of Lemma~\ref{hold1}]
	Let $f\in\cF_\gamma$. By the reproducing property of $h_\gamma$, $f(x)=\langle h_\gamma(x,\cdot),f\rangle_{\cF_\gamma}$ for $f\in\cF_\gamma$. From this, the Cauchy-Schwarz inequality and~(\ref{emb}) with $\phi_\gamma(x)=h_\gamma(x,\cdot)$,
	\begin{equation}\label{holdproof}
	\begin{split}
	\lefteqn{\abs{f(x)-f(x')}
		= \abs{\langle h_\gamma(x,\cdot)-h_\gamma(x',\cdot),f\rangle_{\cF_\gamma}}} \\
	&\quad\quad\le \norm{h_\gamma(x,\cdot)-h_\gamma(x',\cdot)}_{\cF_\gamma} \,\norm{f}_{\cF_\gamma} = \norm{x-x'}_\cX^\gamma\,\norm{f}_{\cF_\gamma},
	\end{split}
	\end{equation}
	proving the lemma.
\end{proof}

The next example illustrates Lemma~\ref{hold1} is sharp.
\begin{example}\label{bm}\rm
	If $\cX=[0,1]$, the FBM-1/2 RKHS consists of all functions $f$ that are absolutely continuous possessing a weak derivative $\dot f$ and satisfying $f(0)=0$. The norm is given by
	\begin{eqnarray*}
		\norm{f}_\cF^2 = \norm[\big]{\dot f}_{\LL^2([0,1])}^2 = \int_{[0,1]} \dot f(x)^2dx \hspace{6mm}\forall f\in\cF.
	\end{eqnarray*}
	See \citet{vz08rep}, Section~10, for a proof.
	It is now straightforward to verify that the function $f$ over $[0,1]$ defined by $f(x)=\sqrt{1-x^2}$ is in $\cF_{1/2}$. This function is H\"older of order $1/2$ but not more.
\end{example}

By~(\ref{repr}), I-prior realizations are finite combinations of basis functions $h_\gamma(x_i,\cdot)$. Hence Lemma~\ref{lem-hol1} in \ref{app-hold} directly implies the following:
\begin{lemma}
	If $\cF$ in~(\ref{regr}) subject to~(\ref{err}) is the FBM-$\gamma$ RKHS, the I-prior has realizations that are H\"older of order $2\gamma$. 
\end{lemma}

\subsection{Regularity}\label{sec-reg}

As seen above, in terms of H\"older smoothness, FBM paths and functions in the corresponding FBM RKHS may differ by an infinitesimally small amount. If we look at a different concept of smoothness, it turns out there is a gap between the two.

Let $\cF$ be a Hilbert space of functions over a set $\cX$ with orthonormal basis $\{g_i\}$.
For $\beta>0$ and $f\in\cF$, consider the squared norm
\begin{align}\label{sobloss}  \norm{f}_\beta^2 = \frac12\sum_{i=1}^\infty f_i^2i^{2\beta}, \end{align}
where $f_i=\langle f,g_i\rangle_{\cF}$.
A function $f$ for which $\norm{f}_\beta<\infty$ is said to have {\em regularity} $\beta$ relative to the basis $\{g_i\}$. 
We have:
\begin{lemma}\label{lem-reg}
	The FBM-$\gamma$ RKHS over $[0,1]$ is regular of order $1/2+\gamma$ relative to the Karhunen-Loeve basis for the FBM-$\gamma$ process. 
\end{lemma}
\proof[Proof of Lemma~\ref{lem-reg}]
It follows from results in \citet{bronski03small} that the FBM-$\gamma$ kernel over $[0,1]$ has eigenvalues $\lambda_i\sim i^{-1-2\gamma}$. 
Now $f\in\cF_{\gamma}$ if and only if $\sum f_i^2/\lambda_i<\infty$ (e.g., Lemma~1.1.1 in \citet{wahba90}).
But then $f_i=o(i^{-1-\gamma})$ as $i\rightarrow\infty$ such that the sum converges.
It follows that $\norm{f}_\beta<\infty$ if and only if $\beta\le1/2+\gamma$.
\endproof

I-prior paths are finite dimensional and hence have infinite regularity. More interestingly, we can may consider asymptotic regularity.
A series of functions $a_1,a_2,\ldots$ can be defined to have asymptotic regularity $\beta$ if $\lim_{n\rightarrow\infty}\norm{a_n(x)}_\beta<\infty$.
From \citet{bronski03small}, with $\cX=[0,1]$, $h_\gamma$ has a Mercer expansion
$h_\gamma(x,x') = \sum_{i=1}^\infty \lambda_i g_i(x)g_i(x')$ 
where $\lambda_i\sim i^{-1-2\gamma}$ and the $g_i$ form the Karhunen-Loeve basis for FBM-$\gamma$. Hence I-prior paths can be written as
\[  f_n(x) = \sum_{j=1}^nh(x,x_j)w_j = \sum_{i=1}^\infty\lambda_iu_ig_i(x) \]
where $u_i=\sum_{j=1}^ng_i(x_j)w_j$. If the errors in~(\ref{regr}) are i.i.d.\  normal, the $w_i$ are i.i.d.\  normal under the I-prior.
Under some conditions on the $x_i$, and assuming the $g_i$ have a common bound (unfortunately we have no proof of this but experimental results support the assertion), we can verify that $\lim_{n\rightarrow\infty}\norm{f_n(x)/n}_\beta<\infty$ a.s.\ for $\beta\le 1+2\gamma$, i.e., I-prior paths multiplied by $1/n$ are a.s.\ asymptotically regular of order $1+2\gamma$. 


\section{Application to real data}\label{sec-applic}

In this section we apply the I-prior methodology to nine real data sets which have been extensively analyzed in the literature, and compare performance to published methods as a well as to Tikhonov regularization. We analyze one functional regression data set and eight classification problems, and obtain competitive performance of the I-prior methodology.

\subsection{Model assumptions and estimation of hyperparameters}\label{sec-estdetails}

\begin{figure}[tbp]
	\centering
	\includegraphics[width=70mm]{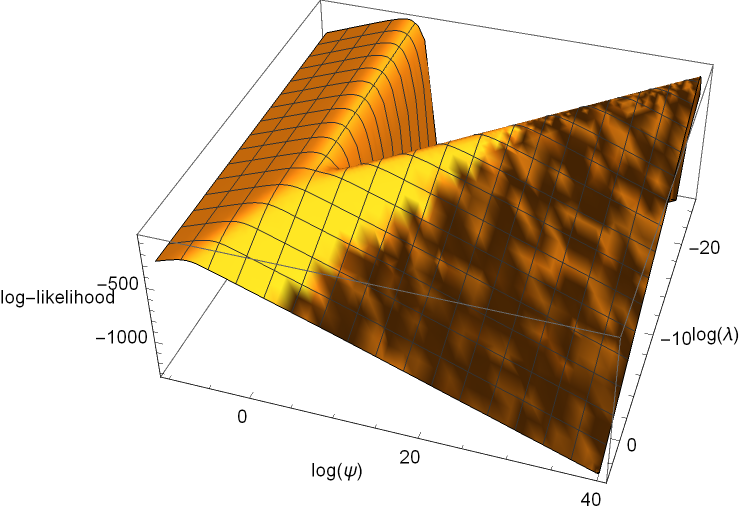}\hspace*{5mm}
	\includegraphics[width=70mm]{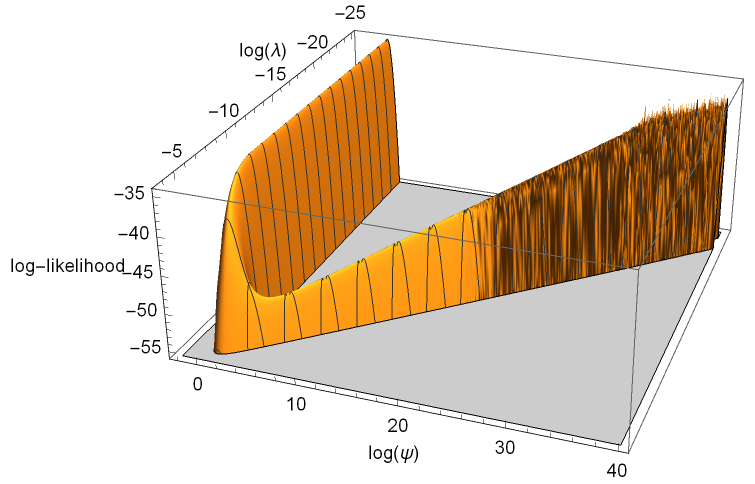}
	\caption{Marginal log likelihood for Eye State data (global view on left and zoom of very top of graph on right). Darker areas reflect a nonsmooth surface due to numerical errors in evaluating the log-likelihood. There are local maxima on each of the two ridges, with the maximum on the diagonal ridge leading to the best predictive performance. It can be seen that the local maximum on the diagonal ridge is numerically hard to find, in fact, it is hard to establish what the global maximum is. Fortunately, predictive performance is near-identical anywhere on the diagonal ridge with say $\log(\psi)>6$, so for the purpose of prediction there is no need to find the actual global maximum.}
	\label{fig-marglik}
\end{figure}

For the real data examples below, we assumed model~(\ref{regr}) subject to i.i.d.\  $N(0,\psi^{-1})$ errors, and for the set of regression functions $\cF$ we used the canonical RKHS of linear functions and the FBM-$1/2$ RKHS. We also made some limited use of the FBM-$\gamma$ RKHS where the Hurst coefficient $\gamma$ was estimated, and the squared exponential RKHS~(\ref{sqexp}) with $\xi=1$ where $\sigma$ was estimated. 
For Tikhonov regularization, we only used the canonical and FBM-1/2 RKHSs, which leave just the scale parameter parameter $\lambda$ in~(\ref{reg}) to be estimated, which we did using generalized cross-validation.

For the I-prior methodology, $\lambda$, $\psi$, and possibly $\gamma$ or $\sigma$ were to be estimated. 
We first discuss estimation of $\lambda$ and $\psi$, which we did by maximizing the marginal likelihood~(\ref{marglik}) (we also tried minimizing various cross-validation criteria, but this gave worse performance). Maximum likelihood estimation was not straightforward for two reasons: the possible occurrence of multiple local maxima, and numerical difficulties in evaluating the likelihood. 
A typical situation is as pictured in Figure~\ref{fig-marglik}: the likelihood has two ridges, one parallel to the $\log(\lambda)$ axis, and one running diagonally across the graph. We found empirically that each ridge may have a local maximum, and there may be a local maximum on or near the cusp as well. Usually if there was a local maximum on the diagonal ridge, we were not able to find it because it was too difficult to numerically evaluate the likelihood. In some cases, particularly for the canonical kernel (e.g., for the Hill-Valley data below), we could not get a decent estimate of any part of the diagonal ridge. Fortunately, in most cases, it was only necessary to be able to estimate the part of the diagonal ridge near the cusp, as predictive performance did not noticeably change moving up the ridge. 
Some more details are given in the caption of Figure~\ref{fig-marglik}.
We selected the local maximum (or near-local-maximum on the diagonal ridge) that gave the smallest cross-validation error. 

As a result, in most cases we could not determine the value of the maximum marginal likelihood, and in these cases it was impossible to find the maximum likelihood estimator of the Hurst coefficient $\gamma$. Instead, for different values of $\gamma$ we estimated $\lambda$ and $\psi$ using maximum likelihood, and selected the value of  $\gamma$ which minimized cross-validation error. As this was quite time consuming, particularly due to the difficulty of finding the maximum likelihood estimators of $\lambda$ and $\psi$ for different values of $\gamma$, we omitted estimation of $\gamma$ from the simulations below.

Notwithstanding some philosophical advantages of a fully Bayes approach compared to an empirical Bayes one, it does not appear to be the case that a fully Bayes approach would alleviate the aforementioned problems. In a fully Bayes setting, we may expect the posterior to be multimodal, so the posterior mean would be an inadequate summary measure, and the posterior mode may be more appropriate. However, finding it may be computationally difficult using Monte Carlo methods. Furthermore, unlike the marginal likelihood, the posterior of $f$ cannot easily be visualized, and it may be difficult to find the true mode among multiple local modes. The numerical difficulties we encountered in evaluating the likelihood and the posterior (with hyperparameters substituted by their maximum likelihood estimators) will likely be replicated in the posterior in a fully Bayes approach.

\subsection{Motivation for use of FBM RKHS}

As explained in more detail in Section~\ref{sec-fbm}, the use of the I-prior methodology is particularly attractive if $\cF$ is a fractional Brownian motion (FBM) RKHS over a Euclidean space (which has as a special case the aforementioned centered Brownian motion). FBM {\em process paths} are non-differentiable and, having H\"older smoothness ranging between 0 and 1, an FBM process prior for the regression function may be too rough for many applications. 
In contrast, functions in the FBM {\em RKHS} are (weakly) differentiable if the Hurst coefficient is at least 1/2 and have minimum H\"older smoothness ranging from 0 to 2. This wide range of smoothnesses make it an attractive general purpose function space for nonparametric regression. 
Another advantage is that it allows us to do multivariate smoothing with just one or two parameters to be estimated: either only the scale parameter $\lambda$, while using a default setting of, say, 1/2 for the Hurst coefficient, or both the scale parameter and the Hurst coefficient. This is in contrast with standard kernel based smoothing methods, which require a scale parameter and at least one kernel hyperparameter to be estimated. For example, if we use the exponential kernel
\begin{equation}\label{sqexp}
r(x,x') = \exp\Big(-\frac{\norm{x-x'}^{2\xi}}{2\sigma^2}\Big), 
\end{equation}
the scale parameter $\lambda$, the smoothness parameter $\xi$ (somewhat analogous to the Hurst coefficient), and a `variance' parameter $\sigma^2$ need to be estimated. Default settings $\xi=1$ or $\xi=2$ could be used to reduce the number of free parameters to two. 

Besides smoothing with FBM kernels, the only other smoothing method that we are aware of that requires only a single hyperparameter to be estimated is thin plate spline smoothing. However, for larger dimensions, thin plate splines seem to be harder to interpret and implement.

\subsection{Regression with a functional covariate}

\begin{figure}[tbp]
	\centering
	{\includegraphics[width=80mm]{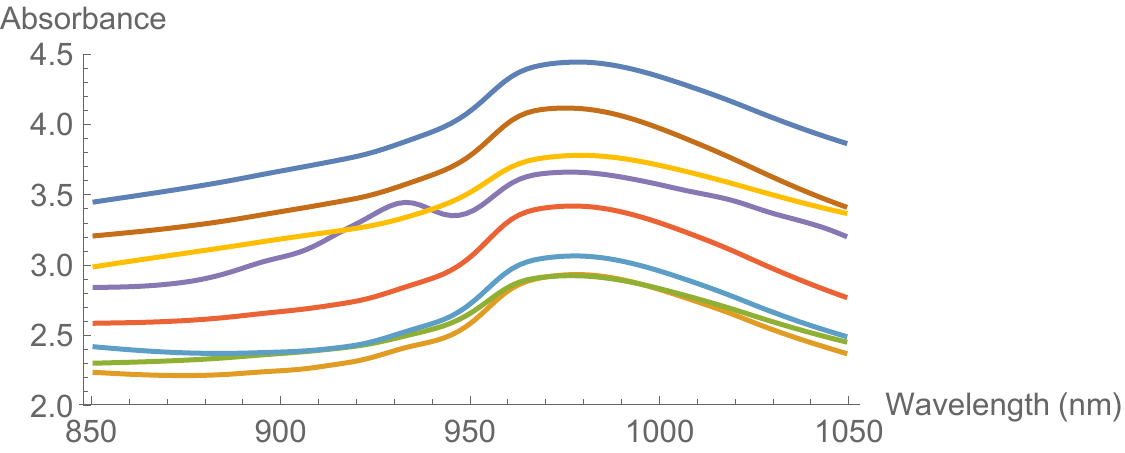}}
	\caption{Sample of spectrometric curves used to predict fat content of meat}\label{fig-tecator}
\end{figure}

We illustrate the prediction of a real valued response when one of the covariates is a function using a widely analysed data set used for quality control in the food industry.
The data consist of measurements on a sample of $215$ pieces of finely chopped meat. The response variable is fat content, and the covariate is light absorbance for 100 different wavelengths.
The absorbance curve can be considered a `functional' variable (see a sample of such curves plotted in Figure~\ref{fig-tecator}).
For more details see http://lib.stat.cmu.edu/datasets/tecator and \citet{thodberg96}. 
Our aim is to predict fat content from the 100 measurements of absorbance.
The first 172 observations in the data set are used as a training sample, and the remaining 43 observations are used as a test sample (following Thodberg's original recommendation).

Many different methods have been applied in the literature to the data set, estimating a model using the training sample and evaluating its performance using the test sample.
One of the best results was achieved early on by~\citet{thodberg96}, who used neural networks on the first 10 principal components and achieved a test mean squared error of $0.36$.
The best test error performance we found was by \citet{vw00} who achieved an error rate of $0.34$, also using neural networks on the principal components.
More recently various other statistical models have been tried on the data set, see  Table~\ref{tbl-spect} for a summary. In spite of their lesser performance compared to neural networks, the interest of these methods is that they do not rely on an a priori data reduction in terms of the main principal components.

\begin{table}[t]
	\centering
	\small
	\begin{tabular}{lrr}\hline
		
		Method                                         & \multicolumn{2}{c}{RMSE} \\
		& Training & Test \\ \hline
		Global constant model                          & 12.50 & 13.3 \\ \hline
		Neural network \citep{vw00}                     &  & 0.34 \\
		Kernel smoothing \cite[Section 7.2]{fv06}      &  &  1.85 \\
		Double index model \citep{chm11}                &  & 1.58 \\
		Single index model \citep{gv14}                 &  & 1.18 \\
		Sliced inverse regression \citep{ll14}          &  & $0.90$ \\ 
		MARS \citep{zyh14}                              &  & 0.88 \\
		Partial least squares \citep{zyh14}             &  & 1.01 \\
		CSEFAM \citep{zyh14}                            &  & 0.85 \\
		\hline
		Tikhonov regularization (linear)               & 3.32 & 3.54 \\
		Tikhonov regularization (FBM-1/2 kernel)       & 4.32 & 4.54 \\ \hline
		I-prior (linear)                                & 2.82 & 3.15 \\
		I-prior (FBM RKHS with $\gamma=0.5    $)                & 0.00 & 0.67 \\
		I-prior (FBM RKHS with $\hat\gamma=0.98$)               & 0.00 & 0.57 \\
		I-prior (squared exponential RKHS, $\hat\sigma=0.0079$) & 0.35 & 0.58 \\
		\hline
	\end{tabular}
	\caption{RMSEs for predicting fat content from spectrometric functional covariate (see Figure~\ref{fig-tecator}): previously published results, Tikhonov regularization, and I-prior methodology.
	}
	\label{tbl-spect}
\end{table}

The $i$th spectral curve is denoted $x_i$, with $x_i(t)$ denoting the absorbance for wavelength~$t$. To be able to estimate a linear or smooth effect using the canonical or FBM RKHSs, an appropriate inner product for the $x_i$ needs to be found. From Figure~\ref{fig-tecator} it appears the curves are differentiable, and it seems reasonable to assume the $x_i$ lie in a Sobolev-Hilbert space $\cX$ with inner product
\[  \langle x,x'\rangle_\cX^2 = \int \dot x(t)\dot x'(t) dt. \]

A linear effect of the spectral curve on fat content can be modelled using the canonical RKHS over $\cX$. We see in Table~\ref{tbl-spect} that both Tikhonov regularization and the I-prior give a poor performance, with test RMSEs of 3.54 and 2.89, respectively. Next we fitted a smooth dependence of fat content on spectrometric curve using the \fbm\ RKHS. As seen in the table, Tikhonov regularization performs very poorly. We tried various values of the Hurst coefficient, but all give worse results than the linear model.
On the other hand, the I-prior performs rather well for different RKHSs, including the FBM and the squared exponential ones. We had some convergence problems so could not get the ML estimator of $\gamma$, the Hurst coefficient for the FBM RKHS, so instead estimated it by minimizing the cross-validation error (10-fold cross-validation gave $\hat\gamma=0.98$). For the squared exponential RKHS we did manage to find the ML estimator $\hat\sigma$ of $\sigma$, and it is given in Table~\ref{tbl-spect}.

Instead of fat content, protein content can be predicted from the spectral curve.
With the I-prior based on a smooth dependence of protein content on the spectral curve we obtained an RMSE of 0.52, using a local (non-global) maximum likelihood estimate of the Hurst coefficient, $\hat\gamma=0.997$. This improves on \citet{zyh14} who obtained an RMSE of 0.85.

\subsection{Classification}

We now apply the I-prior methodology to classification problems, assuming model~(\ref{regr}) with $y_i\in\{0,1\}$ denoting the class label of observation $i$, and $x_i\in\mR^p$ a $p$-dimensional covariate. A newly observed unit $n+1$ with covariate value $x_{n+1}$ is classified into class $0$ if $\hat f(x_{n+1})<0.5$ and into class 1 if $\hat f(x_{n+1})>0.5$.

An extensive analysis of eight data sets with sixteen different methods has recently been done by \citet{cs17}. The methods are the following: linear and quadratic discriminant analysis (LDA and QDA), $k$ nearest neighbours ($k$nn), \citeauthor{cs17}'s random projection version of these methods (RP-LDA, RP-QDA and RP-$k$nn), a single projection version of LDA and $k$nn, random forests (RF, \citeauthor{breiman01}, citeyear{breiman01}), support vector machines (SVMs) with linear and radial kernels, Gaussian process regression with a radial kernel, penalized LDA \citep{wt11}, nearest shrunken centroids \citep{thnc03}, $L_1$ penalized logistic regression \citep{gmc15}, optimal tree ensembles \citep{khan16}, and an ensemble of a subset of $k$nn classifiers \citep{gul16}.

For each data set, random subsamples of sizes between 50 and 1000 were taken, and for each subsample the model was fitted and model based predicted class labels of the remaining data were computed. The number of random random subsamples ranged between 40 and 1000, and the average misclassification percentage for the predictions was computed as well as corresponding standard errors. 

In Table~\ref{tbl-class}, for each data set the best results provided in \citet{cs17} are reproduced, along with results for the I-prior methodology and Tikhonov regularization based on the canonical and FBM-1/2 RKHSs. 
For the Hill-Valley and Mice data, we also included results for the FBM-0.9 RKHS, which dramatically reduced misclassification rates. For most subsamples, $\hat\gamma=0.9$ approximately minimized cross-validation error (it is coincidental that the number is the same for both data sets). For all datasets, initial analyses indicated further improvements of results could be obtained by estimating $\gamma$ rather than using $\gamma=1/2$ or $\gamma=0.9$, but this was too time consuming to carry out. 
For Tikhonov regularization the scale parameter was estimated using generalized cross-validation, and for the I-prior methodology hyperparameters were estimated using a modified maximum likelihood approach (see Section~\ref{sec-estdetails}).
For five out of eight data sets (Eye state, Mice, Hill-Valley, Musk, and Activity recognition), the I-prior methodology gives better results than the best method reported by \citeauthor{cs17}, and for the Gisette data there is a tie for first place with Linear SVMs. For six out of the eight data sets, the I-prior methodology improved on the random projection ensemble results of \citet{cs17}, while only for the Ionosphere data did an ensemble method perform better; however, it is possible that random projection ensembles can further improve the I-prior methodology.
Furthermore, in most instances the I-prior methodology gives better results then Tikhonov regularization, often by a large margin. Part of the reason for this may be that for Tikhonov regularization, only a single parameter is estimated (the scale parameter $\lambda$), while for the I-prior methodology two are estimated ($\lambda$ and $\psi$), giving more flexibility to adapt to the data.

We did not show results for GP regression with the FBM-1/2 and canonical kernels, which gave results comparable to the I-prior methodology, sometimes better, sometimes worse. 

\begin{table}
	\scriptsize
	\centering
	\begin{tabular}{lrrrrrr}\hline
	Method & \multicolumn{3}{c}{Eye state data} & \multicolumn{3}{c}{Ionosphere data}  \\
	& $n=50$ & $n=200$ & $n=1000$   & $n=50$ & $n=100$ & $n=200$       \\ \hline
	Best previous result & 39.0$_{0.4}$ & 26.9$_{0.3}$ & 13.5$_{0.2}$ & \color{red}8.1$_{0.4}$ & \color{red}6.2$_{0.2}$ & \color{red}5.2$_{0.2}$ \\
	--- using method & \scriptsize RP-QDA & \scriptsize RP-$k$nn & \scriptsize RP-$k$nn  & \color{red}\scriptsize RP-QDA  & \color{red}\scriptsize RP-QDA  & \color{red}\scriptsize RP-QDA  \\
	Tikh.\ reg.\ (linear)  & 46.1$_{0.2}$ & 42.7$_{0.2}$ & 37.8$_{0.4}$ & 19.0$_{0.2}$ & 15.2$_{0.2}$ & 13.7$_{0.1}$   \\
	Tikh.\ reg.\ (FBM-1/2) & 46.3$_{0.2}$ & 42.2$_{0.3}$ & 25.8$_{0.5}$ & 19.4$_{0.2}$ & 11.4$_{0.1}$ & 8.4$_{0.1}$   \\
	I-prior (linear) & 46.2$_{0.4}$   &40.0$_{0.2}$& 38.1$_{0.2}$  & 17.3$_{0.2}$&	14.6$_{0.1}$&	13.6$_{0.2}$	\\
	I-prior (FBM-1/2)	&\color{red}37.0$_{0.1}$&\color{red}24.0$_{0.1}$ &\color{red}10.3$_{0.1}$&11.3$_{0.1}$	&7.7$_{0.1}$	&6.1$_{0.1}$  \\[2mm]
	
	\hline
	Method & \multicolumn{3}{c}{Mice data} & \multicolumn{3}{c}{Hill-valley data}  \\
	& $n=200$ & $n=500$ & $n=1000$   & $n=100$ & $n=200$ & $n=500$       \\ \hline
	Best previous & 6.4$_{0.1}$ & 2.2$_{0.1}$ & 0.6$_{0.1}$ & 36.8$_{0.8}$ & 36.5$_{0.9}$ & 32.6$_{1.1}$ \\
	--- using method & \scriptsize LDA  & \scriptsize RP-$k$nn & \scriptsize RP-$k$nn & \scriptsize RP-LDA  & \scriptsize RP-LDA  & \scriptsize RP-LDA   \\
	Tikh.\ reg.\ (linear) & 9.5$_{0.1}$ & 4.3$_{0.1}$ & 3.6$_{0.2}$ & 50.2$_{0.0}$ & 49.8$_{0.1}$ & 35.5$_{0.2}$  \\
	Tikh.\ reg.\ (FBM-1/2)& 25.2$_{0.2}$&12.2$_{0.2}$ & 5.4$_{0.3}$ & 50.3$_{0.0}$ & 50.5$_{0.0}$ & 50.8$_{0.1}$ \\
	I-prior (linear)      & 6.3$_{0.1}$ & 4.3$_{0.1}$ & 3.8$_{0.2}$ &--$_{}$&--$_{}$&	--$_{}$	\\
	I-prior (FBM-1/2)     & 6.8$_{0.0}$ & 1.1$_{0.0}$ & 0.1$_{0.0}$ &47.4$_{0.1}$&	43.4$_{0.2}$&	32.9$_{0.2}$    \\
	I-prior (FBM-0.9)     & \color{red}3.7$_{0.1}$ & \color{red}$0.6_{0.0}$ &\color{red} 0.1$_{0.0}$                   &\color{red}32.5$_{0.1}$&	\color{red}25.4$_{0.1}$&	\color{red}18.9$_{0.2}$    \\[2mm]
	
	\hline
	Method & \multicolumn{3}{c}{Musk data} & \multicolumn{3}{c}{Arrhythmia data}  \\
	& $n=100$ & $n=200$ & $n=500$   & $n=50$ & $n=100$ & $n=200$       \\ \hline
	Best previous & 11.8$_{0.3}$ & 9.7$_{0.2}$ & 7.4$_{0.1}$ & \color{red}30.5$_{0.3}$ & \color{red}26.7$_{0.3}$ & \color{red}22.4$_{0.3}$ \\
	--- using method & \scriptsize RP-$k$nn  & \scriptsize RP-$k$nn & \scriptsize Linear-SVM  & \color{red}\scriptsize RP-QDA  & \color{red}\scriptsize RF  & \color{red}\scriptsize RF  \\
	Tikh.\ reg.\ (linear)  & 13.6$_{0.1}$ & 10.6$_{0.1}$ & 7.7$_{0.0}$ & 44.8$_{0.2}$ & 36.9$_{0.3}$ & 29.1$_{0.1}$   \\
	Tikh.\ reg.\ (FBM-1/2) & 15.1$_{0.1}$ & 11.9$_{0.1}$ & 7.8$_{0.1}$ & 46.3$_{0.2}$ & 39.6$_{0.4}$ & 29.6$_{0.1}$   \\
	I-prior (linear) &15.1$_{0.1}$&	11.5$_{0.2}$&	9.1$_{0.1}$	&		39.1$_{0.2}$ & 33.1$_{0.3}$ & 27.5$_{0.2}$ \\	
	I-prior (FBM-1/2) &\color{red}9.5$_{0.1}$& \color{red}7.0$_{0.1}$&	\color{red}5.0$_{0.1}$ &       31.8$_{0.1}$ & 28.1$_{0.1}$ & 25.5$_{0.1}$ \\[2mm]
	
	\hline
	Method & \multicolumn{3}{c}{Activity recognition data} & \multicolumn{3}{c}{Gisette data}  \\
	& $n=50$ & $n=200$ & $n=1000$   & $n=50$ & $n=200$ & $n=1000$       \\ \hline
	Best previous & 0.11$_{0.02}$ & 0.04$_{0.01}$ & 0.00$_{0.00}$ & \color{red}11.9$_{0.3}$ & \color{red}6.8$_{0.1}$ & 4.5$_{0.1}$ \\
	--- using method & \scriptsize Pen-LDA  & \scriptsize Pen-LDA  & \scriptsize Pen-LDA  & \scriptsize Linear-SVM  & \scriptsize Linear-SVM  & \scriptsize Linear-SVM   \\
	Tikh.\ reg.\ (linear) & 0.28$_{0.00}$ & 0.23$_{0.01}$ & 0.10$_{0.01}$ & 41.2$_{0.5}$ & 11.0$_{0.1}$ & 6.9$_{0.1}$  \\
	Tikh.\ reg.\ (FBM-1/2)& 0.28$_{0.01}$ & 0.25$_{0.01}$ & 0.19$_{0.01}$ & 46.0$_{0.4}$ & 15.4$_{0.4}$ & 7.6$_{0.1}$ \\
	I-prior (linear) & \color{red}0.04$_{0.00}$ & \color{red}0.00$_{0.00}$ & \color{red}0$_{0}$       & 12.3$_{0.1}$ & \color{red}6.9$_{0.1}$ & 4.5$_{0.1}$\\
	I-prior (FBM-1/2)& 0.16$_{0.00}$ & 0.03$_{0.00}$ & 0.00$_{0.00}$ & 14.1$_{0.1}$ & 7.1$_{0.1}$ & \color{red}4.2$_{0.1}$  \\
	\hline
\end{tabular}
	\caption{Average percentage test-set misclassification for eight data sets with standard errors in the subscript. The best previous results are taken from \cite{cs17}, and `RP' refers to their random projection ensemble method. It can be seen that the Tikhonov regularization method performs poorly. The best performer is colored red, which for five out of eight datasets is the I-prior, with a tie for the Gisette data set. 
	Dashes indicate the model could not be fitted due to numerical problems.}
	\label{tbl-class}
\end{table}

\section{Simulation study}\label{sec-sim}

The I-prior methodology is generally applicable, and in this section we attempt to gain some insight into its performance by considering the special case of smoothing over $[0,1]$. We compare the I-prior with Tikhonov regularization and with GPR based on the squared exponential process prior. The main result is that the I-prior estimator has better small sample performance than the Tikhonov regularizer, even in cases most favourable to the latter. Furthermore, compared to the other two methods, the squared exponential prior gives very poor performance for the roughest functions in $\cF$.

The assumed model is given by~(\ref{regr}) where $\cF$ is the centered Brownian motion RKHS over $[0,1]$ given by~(\ref{cbm}), with norm  $\norm{f}_\cF=\big(\int_0^1 \dot f(x)^2dx\big)^{1/2}$, and the errors are i.i.d.\  $N(0,1)$. 
We consider the following three estimators of $f$: 
\begin{itemize}
	
	\item The posterior mean under the I-prior
	
	\item The Tikhonov regularizer, i.e., the minimizer of
	\[  \sum_{i=1}^n(y_i-f(x_i))^2 + \lambda \int_0^1 \dot f(x)^2dx. \]
	
	\item The posterior mean under the squared exponential Gaussian process prior (subsequently referred to as SE estimator), with covariance kernel given by~(\ref{sqexp}) with $\xi=1$. 
	
\end{itemize}
In this case, the Tikhonov regularizer is the posterior mean of the regression function under a centered Brownian motion prior. 
In all cases, we estimate the smoothing parameter by maximum marginal likelihood or the implied marginal likelihood for Tikhonov regularization.

We included the SE estimator as it is commonly used, and, if the scaling parameter $\lambda$ is suitably chosen, it has optimal asymptotic convergence rate for all functions in $\cF$ \citep{vz07}. 
As mentioned above, in the present case, the regularizer of $f$ is its posterior mean under a Brownian motion prior. 
Brownian motion paths have regularity 1/2, while functions in $\cF$ have regularity greater than 1. Hence the sample paths of Brownian motion are `too rough', and the posterior mean (i.e., the Tikhonov regularizer) is expected to undersmooth. 
As shown in Section~\ref{sec-fbm}, the I-prior for $f$ is an integrated Brownian bridge which has regularity 1.5, so should perform well for functions of intermediate smoothness, but not necessarily for very rough or very smooth functions in $\cF$. Similarly, the SE estimator would not necessarily be expected to perform well for non-analytic functions.

The functions in $\cF$ have a wide range of smoothness, ranging from functions which merely have one derivative to analytic functions. Hence, no estimator can be expected to perform well for all functions in $\cF$, but a desirable estimator would perform reasonably across a wide range of smoothnesses. Normally, we would probably desire good performance for the rougher functions in $\cF$.

To assess performance, we simulated regression functions with regularities 1, 1.5, and $\infty$ (see Figure~\ref{fig-paths}). Note that for the simulations it is only necessary to evaluate the simulated functions at $x_1,\ldots,x_n$. With $h$ the covariance kernel of centered Brownian motion given by~(\ref{cbm}), $H$ the matrix with $(i,j)$th element $h(x_i,x_j)$ and $w=(w_1,\ldots,w_n)$ a vector of i.i.d.\  normals, we simulated the following:
\begin{enumerate}
	\item[(a)] `Rough' functions, generated as $\bff=H^{3/4}w$. (In the limit as $n\rightarrow\infty$, these can be shown to have regularity 1 and hence are slightly rougher than the roughest functions in $\cF$.) Due to their roughness, these the functions should most favour Tikhonov regularization.
	\item[(b)] Functions of regularity 1.5, generated as $\bff=Hw$. These are sample paths of the I-prior and this scenario should hence favour the I-prior.
	\item[(c)] Analytic functions (regularity=$\infty$) generated as sample paths of the squared exponential process with $\sigma=0.02$. Clearly, this scenario is expected to favour the squared exponential prior.
\end{enumerate}
We standardized simulated sample paths so that their RKHS norm equals 1, see Figure~\ref{fig-paths} for examples of sample paths.
The centering of the paths means no intercept needs to be estimated, simplifying the simulations.

We measured quality of estimation by the $L_2$ median absolute error (MAE):
\[  \text{MAE}(L_2):=\text{median}(\norm{\hat f-f}_{L_2}). \]
The reason we took median rather than root mean square error was related to robustness. In particular, in a small number of cases, where we did not manage to obtain good convergence. As a result, we worried there might be some bias in the estimation of the mean squared error due to outliers, and took median squared error instead.

In a supplementary online document we also report the MAE based on two other norms, namely
$\text{MAE}(\cF_n):=\text{median}(\norm{\hat f-f}_{\cF_n})$ and $\text{MAE}(\cF):=\text{median}(\norm{\hat f-f}_{\cF})$.
Essentially the same results are seen, but more strongly.

\begin{figure}[tbp]
	\centering
	\includegraphics[width=70mm]{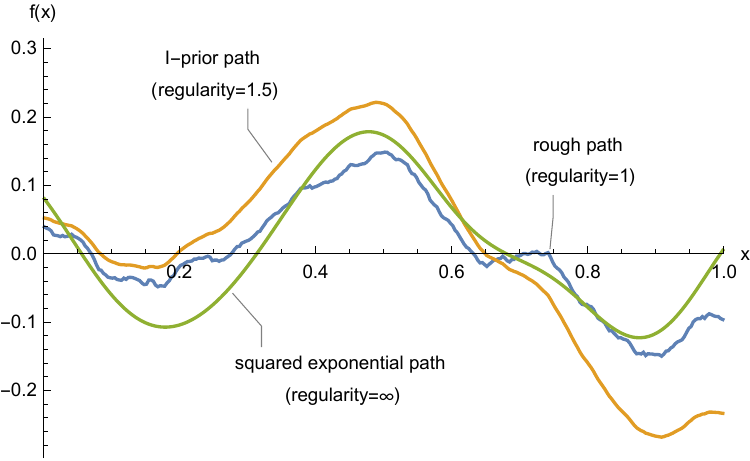}
	\caption{Example sample paths of three different regularities used in simulation. The true regression function is in the centered Brownian motion RKHS, consisting of functions of regularity greater than~1. 
		Paths are centered to integrate to zero and have unit RKHS length.
	}
	\label{fig-paths}
\end{figure}

Further simulation details are as follows. We took a sample size $n=50$ and the $x_i$ equally spaced over $[0,1]$. This sample size makes the computations tractable, and our explorations with other sample sizes showed no essential differences in conclusions.
Hyperparameters were estimated using maximum marginal likelihood. For regularization and the I-prior method, only one hyperparameter needs to be estimated, namely the scale parameter (denoted $\lambda$ in the paper). For the SE estimator, an additional hyperparameter needs to be estimated, namely the parameter $\sigma$ in the formula above. The latter makes the SE estimator significantly more difficult to compute for two reasons: (i) it takes more time to search for a local maximum of the marginal likelihood, and (ii) it is often more difficult to find the global maximum because more starting values need to be tried. As can be seen in Figure~\ref{fig-sim3}, estimation of the SE estimator broke down for very small error standard deviations and rough truths. For all estimators multiple local maxima of the marginal likelihood were sometimes encountered so we used several starting values, so that most of the time we could find the global maximum. However, in particular in some extreme cases (such as very small error standard deviations) we found for some data it could be very difficult to find the global maximum, especially for the SE estimator.
As mentioned, we computed the {\em median} absolute error (MAE) rather than the mean squared error for robustness purposes.

The simulation results are displayed in Figure~\ref{fig-sim3} using log-log plots of the MAE as a function of the error standard deviation.
It is seen that the I-prior method always outperforms regularization, though the advantage of the former is small for the roughest functions in the RKHS (see the subfigures (a)).
For rougher true regression functions in the RKHS, the I-prior estimator outperforms the SE estimator, which breaks down numerically for small errors. For analytic truths, the SE estimator outperforms the I-prior, as was to be expected. 
Overall, because the I-prior method can estimate very smooth functions quite well, but the SE estimator cannot estimate rough functions well, the SE estimator does not seem satisfactory for use in the present case. Furthermore, as mentioned before, the SE estimator is numerically more difficult to find.


\begin{figure}[tbp]
	\centering
	\subfigure[True regression function has regularity 1]
	{\hspace*{0mm}\includegraphics[width=70mm]{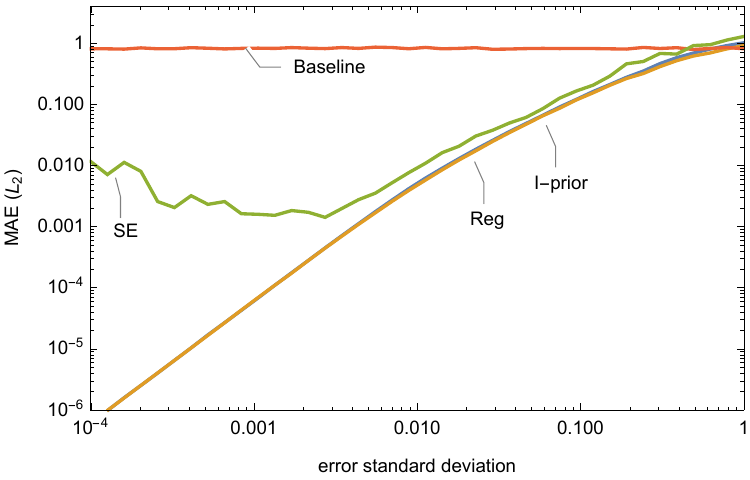}\hspace*{5mm}
		\hspace{2mm}
		{\hspace*{0mm}\includegraphics[width=70mm]{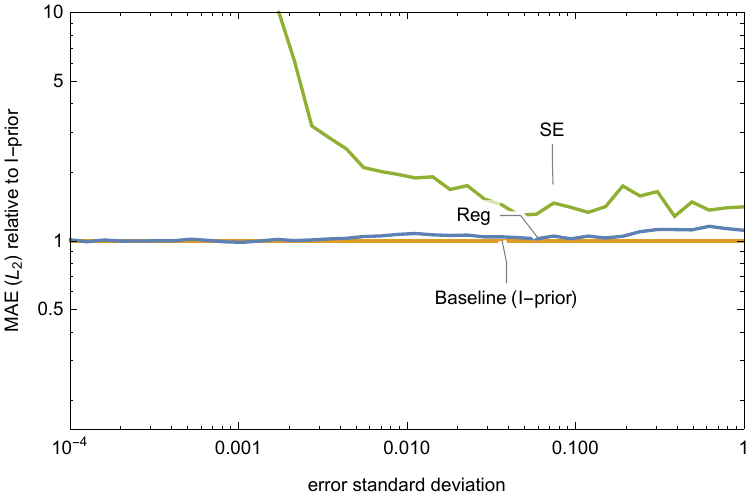}\hspace*{5mm}}}
	\\
	\subfigure[True regression function has regularity 1.5]
	{\hspace*{0mm}\includegraphics[width=70mm]{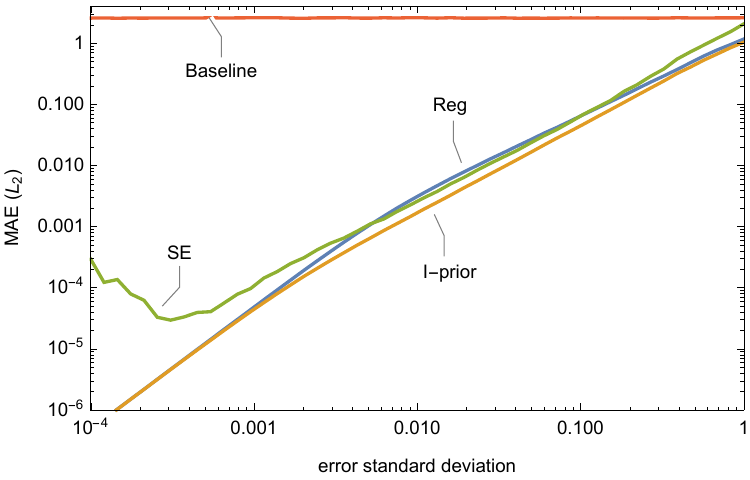}\hspace*{5mm}
		\hspace{2mm}
		{\hspace*{0mm}\includegraphics[width=70mm]{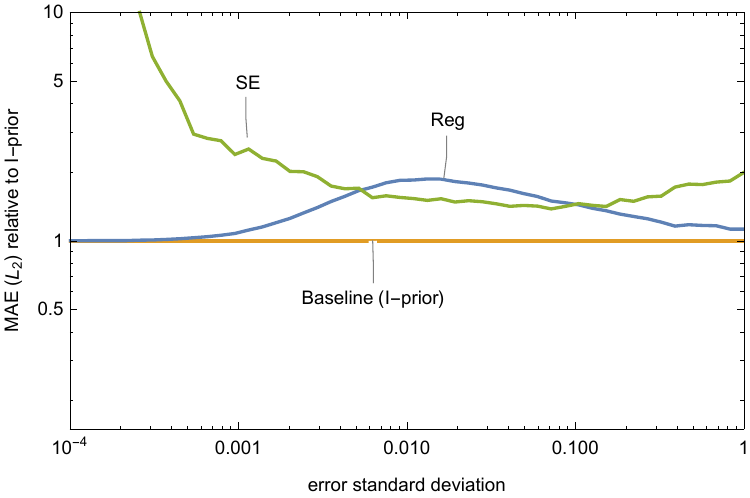}\hspace*{5mm}}}
	\\
	\subfigure[True regression function is a squared exponential Gaussian process path]
	{\hspace*{0mm}\includegraphics[width=70mm]{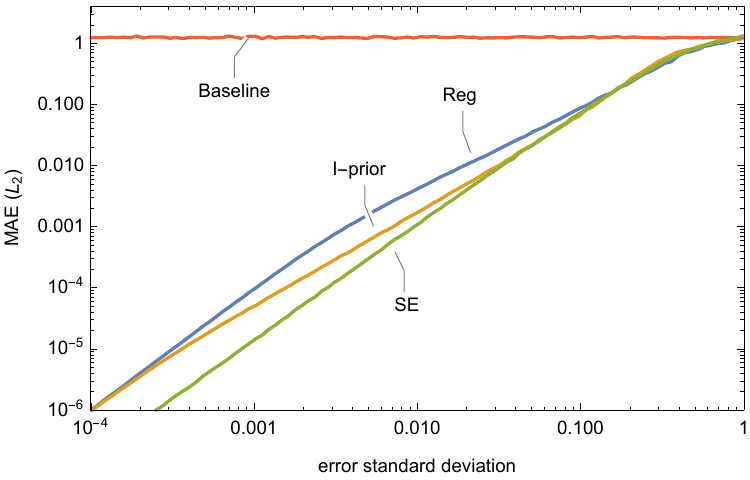}\hspace*{5mm}
		\hspace{2mm}
		{\hspace*{0mm}\includegraphics[width=70mm]{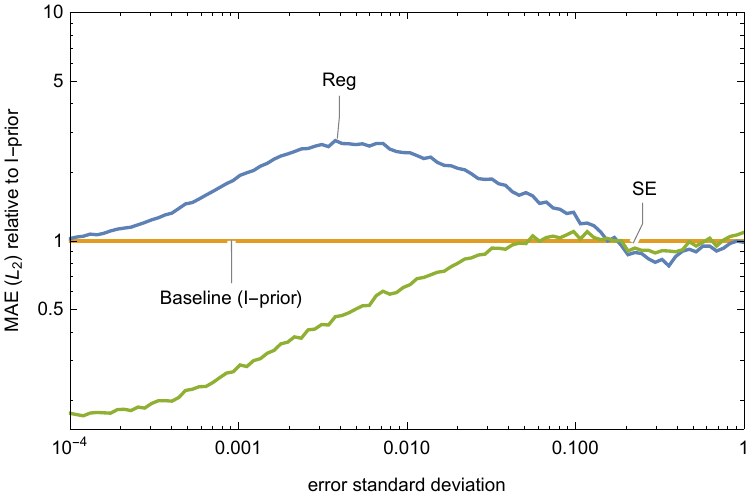}\hspace*{5mm}}}
	\\
	\caption{Panels on left: simulated MAE($L_2$) for Tikhonov regularizer (`Reg'), I-prior estimator (`I-prior'), and SE estimator (`SE'). The baseline level is the MAE if the zero function is fitted.  
		Panels on right: ratio of MAE($L_2$) for regularizer and SE estimator compared to I-prior.
		Model~(\ref{regr}) is assumed with $\cF$ the FBM-1/2 RKHS and i.i.d.\  normal errors. }
	\label{fig-sim3}
\end{figure}

\section{Conclusion}\label{sec-conc}

We introduced the I-prior methodology for estimating a parametric or nonparametric regression function in a likelihood setting. One advantage of the methodology is that, because the I-prior is proper, the posterior mean is an admissible estimator. This is unlike the Tikhonov regularizer, which is the main alternative methodology that can be used in the same setting. Simulations and real data analyses also show better performance of the I-prior methodology compared to Tikhonov regularization. 
Furthermore, the I-prior methodology is as automatic as would seem possible. The user {\em does} need to specify a distribution for the errors and a reproducing kernel for the space of regression functions, but it is hard to envisage how one could proceed without this effort. In practice, suitable choices are often i.i.d.\  or autoregressive errors, and a canonical or FBM-1/2 RKHS, for linear and smooth regression functions, respectively. In the real data examples, these were shown to give good predictive performance. 
The I-prior methodology is general in the following sense: if a regression function $f$ is assumed to lie in a Hilbert space which is not an RKHS, then the Fisher information on $f$ need not exist, in which case it is difficult to see how $f$ could be estimated in a pointwise consistent way.

Further work consists of generalizing the I-prior methodology to multiple, possibly multidimensional covariates. Each covariate then requires a separate reproducing kernel which can be combined using the ANOVA construction \citep{wahba90anova,gw93}.  Each kernel requires a scale parameter which can take values on the real line, and if some are negatively valued, the resulting kernel becomes indefinite generating a {\em reproducing kernel Krein space}. As the Fisher information remains positive definite, the I-prior methodology straightforwardly extends to this case. In this setting, variable selection can be performed.


\appendix

\section{Fisher information, associated distances, and maximum entropy priors}\label{app-fish}

This section puts the I-prior methodology in a somewhat broader context and shows how it can be generalized. We also give a generalization of Zellner's $g$-prior, which we call Rao-Jeffreys prior, being based on the Rao metric and Jeffreys measure. The I-prior and Rao-Jeffreys priors are based on two different Riemannian metrics derived from the Fisher information.

\subsection{Fisher information and distances between probability distributions}\label{app-fishdist}

We first define the Fisher information, and then describe two distances between probability distributions based on it, namely (i) the Rao distance, which is the length of the shortest geodesic in the Riemannian metric induced by the Fisher information matrix, and (ii) the length of the shortest geodesic in the Riemannian metric induced by the {\em inverse} Fisher information matrix. 
The former is parameterization invariant and measures the amount of information between two parameter values.
If the Fisher information does not depend on the parameter of interest, the latter can be easily related to the Cram\'er-Rao bound (see~(\ref{cr})).
Finally, we give the volume measures associated with the distances. The volume measure associated with the Rao distance is well-known to be the Jeffreys measure.


Let $\Theta$ be a Hilbert space with inner product $\langle\cdot,\cdot\rangle_\Theta$, and let $X$ be a random variable with density in the parametric family $\{P(\cdot|\theta)|\theta\in\Theta\}$.
If $P(X|\theta)>0$, the log-likelihood function of $\theta$ is denoted $L(\theta|X)=\log P(X|\theta)$.
Assuming existence, the score is defined as the gradient $\nabla L(\theta|X)$ (see \ref{app-grad} for the definition of the gradient), and the Fisher information $I[\theta]\in\Theta\otimes\Theta$ for $\theta$ as
\[  I[\theta] = -E\big[\nabla^2 L(\theta|X)\,\big|\,\theta\big]. \]
For $b\in\Theta$, denote $\theta_b=\langle\theta,b\rangle_\Theta$.
We define the Fisher information on $\theta_b$ as
\[  I[\theta_b] = \big\langle I[\theta],b\otimes b \big\rangle_{\Theta\otimes\Theta}, \]
and the Fisher information between $\theta_b$ and $\theta_{b'}$ as
\[  I[\theta_b,\theta_{b'}] = \big\langle I[\theta],b\otimes b' \big\rangle_{\Theta\otimes\Theta}, \]
where $\langle\cdot,\cdot\rangle_{\Theta\otimes\Theta}$ is the usual inner product on the tensor product space $\Theta\otimes\Theta$.

We now consider two distances between probability distributions $P(\cdot|\theta)$, $\theta\in\Theta$, based on the Fisher information on $\theta$.
We assume $\Theta$ possesses a finite dimensional parameterization such that the Fisher information $I[\theta]$ for $\theta$ is nonsingular.

The first is the well-known Rao distance $D_\text{Rao}$, defined as the length of the shortest geodesic on the Riemannian manifold whose metric tensor is the Fisher information \citep{rao45,am81,amari85}. As an example, consider the family of multivariate normal distributions with unknown mean $\mu\in\mR^p$ and known covariance matrix $\Sigma$. The Fisher information on $\mu$ is $\Sigma^{-1}$, which does not depend on $\mu$ so the metric is flat, and the Rao distance between distributions indexed by their mean is the Mahalanobis distance and is given by
\[  D_\text{Rao}(\mu,\mu')^2=(\mu-\mu')^\top\Sigma^{-1}(\mu-\mu'). \]
Methods for computing the Rao distance are given by \citet{am81}, and a list of further examples is given by \citet{rao87}.
The Rao distance is invariant to reparameterization, which is an advantage if the parameterization of the model is arbitrary, but may be a disadvantage if the parameterization is not arbitrary, because scale information is lost.


This paper introduces a second distance that depends on the Fisher information, namely the distance $D_I$ defined as the length of the shortest geodesic on the Riemannian manifold with metric tensor the {\em inverse} Fisher information. 



For a distance $D$, define $\nu_D$ to be the associated volume measure. For Euclidean $\theta$, the densities relative to Lebesgue measure are
\begin{align}\label{jef}
\nu_{D_\text{Rao}}(\theta)=\sqrt{\abs[\big]{I[\theta]}}
\end{align}
and
\begin{align}\label{invjef}
   \nu_{D_I}(\theta)=\sqrt{|I[\theta]^{-1}|}.
\end{align}
The measure $\nu_{D_\text{Rao}}$ is the well-known Jeffreys measure or `Jeffreys prior' \citep{jeffreys46}.


\subsection{Maximum entropy distributions}\label{app-maxent}

In a class of distributions, we may consider the one maximizing entropy. Such maximum entropy distributions can be thought of as the `least informative' concerning a parameter of interest, and may hence by useful as so-called noninformative prior distributions in Bayesian inference.

Let $(\Theta,D)$ be a metric space and let $\nu=\nu_D$ be a volume measure over $\Theta$ induced by $D$ (e.g., Hausdorff measure).
Denote by $\pi$ a density on $\Theta$ relative to $\nu$, i.e., if $\theta$ is a random variable with density $\pi$, then for any measurable subset $A\subset\Theta$, $\Pr(\theta\in A)=\int_A\pi(t)\nu(dt)$.
With $\theta_0\in\Theta$, let $\Pi_D$ be the class of distributions $\pi$ such that
\[  E_\pi D(\theta,\theta_0)^2 = \text{constant}. \]
The entropy of $\pi$ relative to $\nu$ is
\[  \cE(\pi) = \int_\Theta \pi(t)\log\pi(t)\nu(dt). \]
Standard variational calculus shows that, if it exists, the density maximizing $\cE(\pi)$ subject to the constraint that $\pi\in\Pi_D$ is given by
\[  \pi_{D}(t) = \frac{e^{-\lambdaexp D(t,\theta_0)^2}}{\int_\Theta e^{-\lambdaexp D(s,\theta_0)^2}\nu(ds)} \propto e^{-\lambdaexp D(t,\theta_0)^2},  \]
where $\lambda$ is a function of the above constant.
This distribution can be thought of as maximizing `uncertainty' subject to the constraint that the expected squared distance of the random variable $\theta$ from some `best guess' $\theta_0$ is fixed.
If $(\Theta,D)$ is a Euclidean space, $\nu$ is a flat (Lebesgue) measure and $\pi_D$ is a multivariate normal density.

We now give three such maximum entropy priors, including two based on the two distances defined in the previous section and derived from the Fisher information.
Suppose $\Theta$ is a finite dimensional affine subspace of a Hilbert space with norm $\norm{\cdot}_\Theta$. Set $D_\Theta(\theta,\theta')=\norm{\theta-\theta'}_\Theta$.
Defining $\pi_\text{reg}:=\pi_{D_\Theta}$ we obtain
\begin{align*}
\pi_\text{reg}(\theta) \propto e^{-\lambdaexp \norm{\theta-\theta_0}^2},
\end{align*}
which is a Gaussian density as the volume measure $\nu_{D_\Theta}$ is flat.
The subscript `reg' refers to regularization, because the posterior mode based on $\pi_\text{reg}$ is the usual regularizer of $\theta$, based on minimizing
\[  \theta \mapsto -\log\Pr(X|\theta) + \lambdaexp \norm{\theta-\theta_0}^2. \]
Alternatively, we can define a prior based on the aforementioned distances $D_\text{Rao}$ and $D_I$, which are based on the Fisher information.
We denote $\pi_{RJ}:=\pi_{D_\text{Rao}}$ and refer to this as the {\em Rao-Jeffreys prior}, being based on the Rao distance and the Jeffreys measure~(\ref{jef}).
We denote $\pi_I:=\pi_{D_I}$ and refer to this as the {\em I-prior}.
The prior densities relative to Lebesgue measure are given by
\begin{align*}
\pi_\text{RJ}(f_w) \propto \sqrt{\abs[\big]{I[\theta]}}\,e^{-\lambdaexp D_\text{Rao}(\theta,\theta_0)^2} 
\hspace{5mm}\text{and}\hspace{5mm}
\pi_I(f_w) \propto \sqrt{\abs[\big]{I[\theta]^{-1}}}\,e^{-\lambdaexp D_I(\theta,\theta_0)^2}.
\end{align*}
The I-prior can be generalized to infinite dimensional spaces, as done in this paper, but the Rao-Jeffreys prior cannot.


\section{H\"older smoothness of FBM RKHS basis functions}\label{app-hold}

For Lemma~\ref{lem-hol1} below we need the following two lemmas on geometric inequalities.


\begin{lemma}\label{lem-tri}
	[A triangle inequality] Let $0\le\gamma\le 1$ and let $d_{ij}$ denote pairwise distances between points $i$ and $j$ in a metric space. Then
	\begin{eqnarray}\label{tri}
	d_{12}^\gamma \le d_{13}^\gamma + d_{23}^\gamma .
	\end{eqnarray}
\end{lemma}
\proof[Proof]
Let $r(z)=(1+z^\gamma)-(1+z)^\gamma$. Then $\dot r(z)=\gamma/z^{1-\gamma}-\gamma/(1+z)^{1-\gamma}>0$ for $z>0$. Hence, $r$ is increasing and since $\lim_{z\rightarrow 0}r(z)=0$, $r(z)>0$ for $z>0$.
Thus, $(1+z)^\gamma\le 1+z^\gamma$ for $z>0$.
By the triangle inequality and this result,
\[  d_{12}^\gamma \le (d_{13}+d_{23})^\gamma = d_{13}^\gamma(1+d_{23}/d_{13})^\gamma \le d_{13}^\gamma(1 + (d_{23}/d_{13})^\gamma) = d_{13}^\gamma + d_{23}^\gamma.
\]
\endproof

\begin{lemma}\label{lem-par} [A parallelogram inequality]
	Let $x$ and $x'$ be points in a Hilbert space $(\cX,\norm{\cdot})$. Then for $0\le\gamma\le 1$,
	\begin{eqnarray}\label{par}
	\norm{x+x'}^{2\gamma} + \norm{x-x'}^{2\gamma} \le 2\norm{x}^{2\gamma} + 2\norm{x'}^{2\gamma},
	\end{eqnarray}
	with equality if $\gamma=1$.
\end{lemma}
\begin{proof}
	With $\phi$ a metric embedding into a Hilbert space satisfying~(\ref{emb}), let $A=\phi_{\gamma}(0)$, $B=\phi_{\gamma}(x)$, $C=\phi_{\gamma}(x+x')$ and $D=\phi_{\gamma}(x')$. Denoting the length of the line segment between $A$ and $B$ by $AB$, and so on, we have $AB=CD=\norm{x}^{\gamma}$, $AD=BC=\norm{x'}^{\gamma}$, and $BD=\norm{x-x'}^{\gamma}$.
	With $E=B+D-A$, the points $A,B,E,D$ form a parallelogram, and the parallelogram law gives
	\begin{equation}\label{par3}  AE^2 + BD^2 = 2AB^2 + 2AD^2 \end{equation}
	Let $M=(A+E)/2=(B+D)/2$ be the midpoint of the parallelogram.
	By a symmetry argument, $AM=CM$, and the triangle inequality gives $AC\le AM+CM=AE$.
	Hence, using~(\ref{par3}), $AC^2 + BD^2 \le 2AB^2 + 2AD^2$, which is equivalent to~(\ref{par}) and completes the proof.
\end{proof}

\noindent
For $\gamma=1$,~(\ref{tri}) is the usual triangle inequality and~(\ref{par}) is the parallelogram law.

\begin{lemma}\label{lem-hol1}
	For $0<\gamma<1$ and $x_0\in\cX$, the function $h_\gamma(x_0,\cdot)$ over $\cX$ is H\"older of order~$2\gamma$.
\end{lemma}
\begin{proof}
	For $0<\gamma\le 1/2$ we have 
	\begin{align*}
	\abs{h_\gamma(x_0,x)-h_\gamma(x_0,x')} 
	&= 
	\frac12\abs[\big]{\norm{x_0-x}_\cX^{2\gamma} - \norm{x_0-x'}_\cX^{2\gamma} - \norm{x}_\cX^{2\gamma} + \norm{x'}_\cX^{2\gamma}} \\
	&\le 
	\frac12\abs[\big]{\norm{x_0-x}_\cX^{2\gamma} - \norm{x_0-x'}_\cX^{2\gamma}} + \frac12\abs[\big]{ \norm{x}_\cX^{2\gamma} - \norm{x'}_\cX^{2\gamma}} \\
	&\le \norm{x-x'}_\cX^{2\gamma},
	\end{align*}
	where the last inequality is due to the triangle inequality given in Lemma~\ref{lem-tri}. 
	For $1/2<\gamma<1$,
	\begin{align*}
	\lefteqn{\abs{h_\gamma(x_0,x-t)-2h_\gamma(x_0,x)+h_\gamma(x_0,x+t)} }\\
	&= 
	\frac12\abs[\Big]{\norm{x_0-x+t}_\cX^{2\gamma} - 2\norm{x_0-x}_\cX^{2\gamma} + \norm{x_0-x-t}_\cX^{2\gamma} - \norm{x-t}_\cX^{2\gamma} + 2\norm{x}_\cX^{2\gamma} - \norm{x+t}_\cX^{2\gamma}} \\
	&\le 
	\frac12\abs[\Big]{\norm{x_0-x+t}_\cX^{2\gamma} - 2\norm{x_0-x}_\cX^{2\gamma} + \norm{x_0-x-t}_\cX^{2\gamma}} +
	\frac12\abs[\Big]{\norm{x-t}_\cX^{2\gamma} - 2\norm{x}_\cX^{2\gamma} + \norm{x+t}_\cX^{2\gamma}} \\
	&\le \norm{x-x'}_\cX^{2\gamma},
	\end{align*}
	where the last inequality is due to the parallelogram inequality given in Lemma~\ref{lem-par}.
\end{proof}

\section{The gradient}\label{app-grad}

Let $(\cH,\langle\cdot,\cdot\rangle)$ be an inner product space and consider a function $g:\cH\rightarrow\mR$.
Denote the directional derivative of $g$ in the direction $s\in\cH$ by $\nabla_s g$, that is,
\begin{eqnarray}\label{dirder}
\nabla_s g(x) = \lim_{\delta\rightarrow 0}\frac{g(x+\delta s)-g(x)}{\delta}.
\end{eqnarray}
The {\em gradient} of $g$, denoted by $\nabla g$, is the unique vector field satisfying
\[  \langle \nabla g(x),s\rangle = \nabla_s g(x) \hspace{6mm}\forall x,s\in\cH. \]

\section{Duality between AR(1) and MA(1) processes}\label{app-ar}

Let $\alpha$ be a real number. Let $u=(u_1,\ldots,u_n)$ be the AR(1) process with parameter $\alpha$ defined by
\begin{eqnarray*}
	u_{1} = \epsilon_1 \hspace{10mm}
	u_{i} = \alpha u_{i-1} + \epsilon_i  \hspace{2mm}(i=2,\ldots,n),
\end{eqnarray*}
where the $\epsilon_i$ are i.i.d.\  $N(0,\sigma^2)$.
Let $v=(v_1,\ldots,v_n)$ be the MA(1) process with parameter $-\alpha$ defined by
\begin{eqnarray*}
	v_{i} = \zeta_i - \alpha\zeta_{i+1}  \hspace{2mm}(i=1,\ldots,n-1)\hspace{10mm}
	v_{n} = \zeta_n,
\end{eqnarray*}
where the $\zeta_i$ are i.i.d.\  $N(0,\sigma^{-2})$.
Denote the covariance matrices of $u$ and $v$ by $V_u$ and $V_v$.
\begin{lemma}\label{lem-ar}
	$V_v=V_u^{-1}$.
\end{lemma}
\proof
Write $\epsilon=(\epsilon_1,\ldots,\epsilon_n)^\top$ and $\zeta=(\zeta_1,\ldots,\zeta_n)^\top$. Then
$u=A\epsilon$ and $v=B\zeta$ where $A$ and $B$ have elements
\[  a_{ij} = \left\{ \begin{array}{cl} 0 & i<j \\ \alpha^{i-j} & i\ge j \end{array}  \right.
\hspace{10mm}
b_{ij} = \left\{ \begin{array}{rl} 1 & i=j \\ -\alpha & i=j-1 \\ 0 & \text{otherwise} \end{array}  \right. . \]
Direct multiplication shows that $AB^\top=A^\top B=I$. Now $V_u=\sigma^2AA^\top$ and $V_v=\sigma^{-2}BB^\top$, so $V_uV_v=I$, which is the desired result.
\endproof

\bibliographystyle{elsarticle-harv}
\bibliography{stats}


\end{document}


\begin{frontmatter}
		
		
		
		\title{Supplementary material for the article ``Regression with I-priors''}
		
		
		
		\begin{abstract}
Simulations complementing the ones in Section~7 in the main article are presented. 
		
		\end{abstract}

	\end{frontmatter}
	
	

\section{Further simulations}\label{app-sim}

In the main article, we considered median absolute errors (MAEs) based on $L_2$ loss, summarized in Figure 5 there. Here we report to additional figures, visualizing the MAE based on two other norms, namely
\begin{itemize}
	\item $\text{MAE}(\cF_n):=\text{median}(\norm{\hat f-f}_{\cF_n})$
	\item $\text{MAE}(\cF):=\text{median}(\norm{\hat f-f}_{\cF})$
\end{itemize}
The simulation results are displayed in Figures~\ref{fig-sim1} and \ref{fig-sim2} using log-log plots of the MAE as a function of the error standard deviation.
As before, it is seen that the I-prior method always outperforms regularization, though the advantage of the former is small for the roughest functions in the RKHS (see the subfigures (a)).
Note that with respect to MAE($\cF_n$) (Figure~\ref{fig-sim1}) and not too small errors, regularization performs worse that a global constant fit (the horizontal `Baseline'). 
For rougher true regression functions in the RKHS, the I-prior estimator outperforms the SE estimator (the posterior mean under a square exponential process prior), which breaks down numerically for small errors. For analytic truths, the SE estimator outperforms the I-prior, as was to be expected.

\begin{figure}[tbp]
	\centering
	\subfigure[True regression function has regularity 1]
	{\hspace*{0mm}\includegraphics[width=70mm]{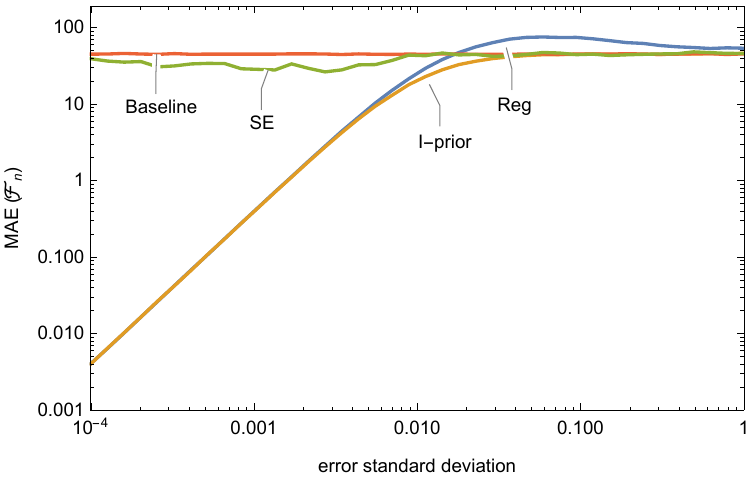}\hspace*{5mm}
		{\hspace*{0mm}\includegraphics[width=70mm]{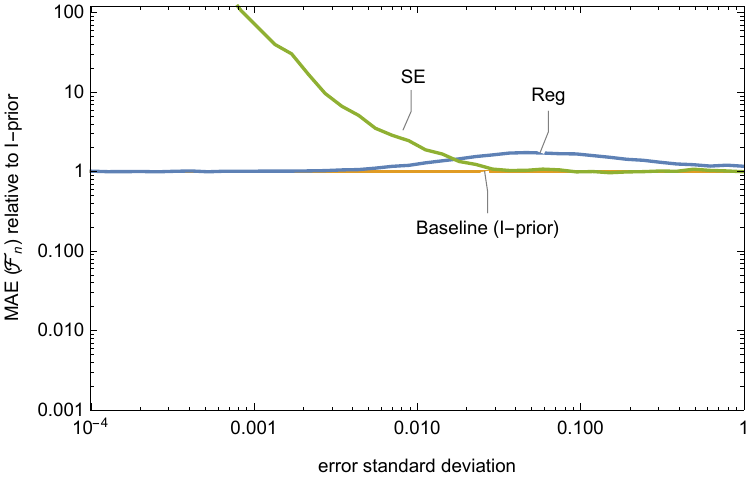}
	}}
	\\
	\subfigure[True regression function has regularity 3/2]
	{\hspace*{0mm}\includegraphics[width=70mm]{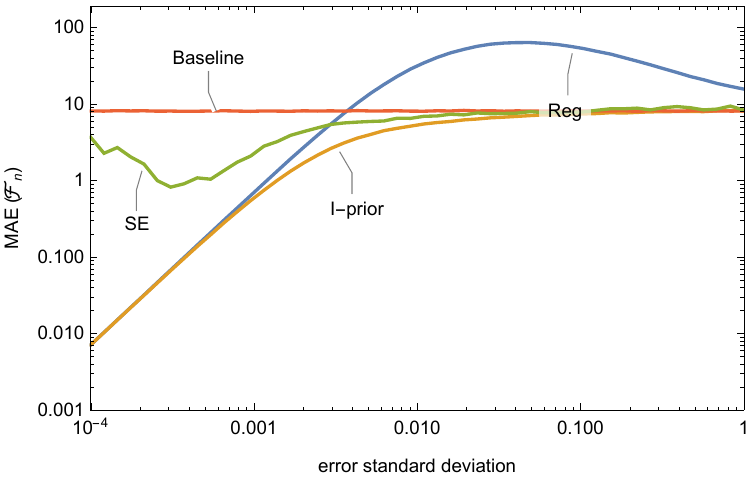}\hspace*{5mm}
		{\hspace*{0mm}\includegraphics[width=70mm]{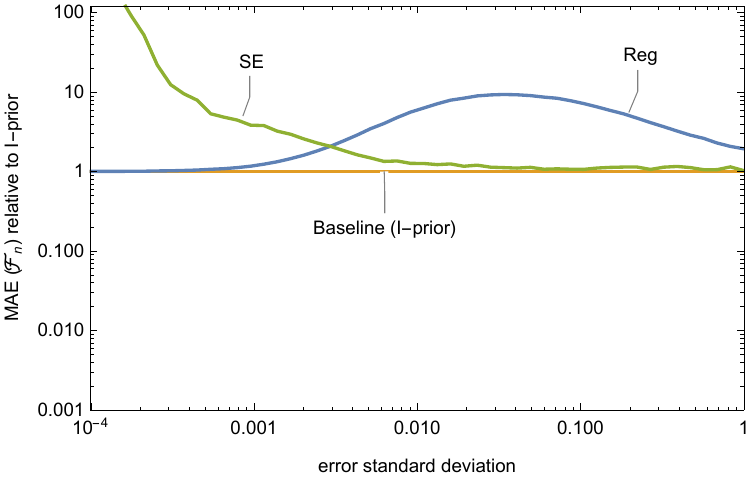}}}
	\\
	\subfigure[True regression function is a squared exponential Gaussian process path]
	{\hspace*{0mm}\includegraphics[width=70mm]{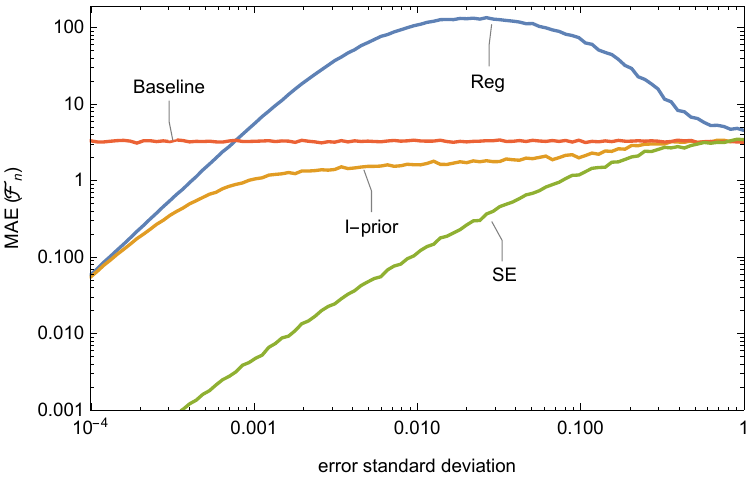}\hspace*{5mm}
		{\hspace*{0mm}\includegraphics[width=70mm]{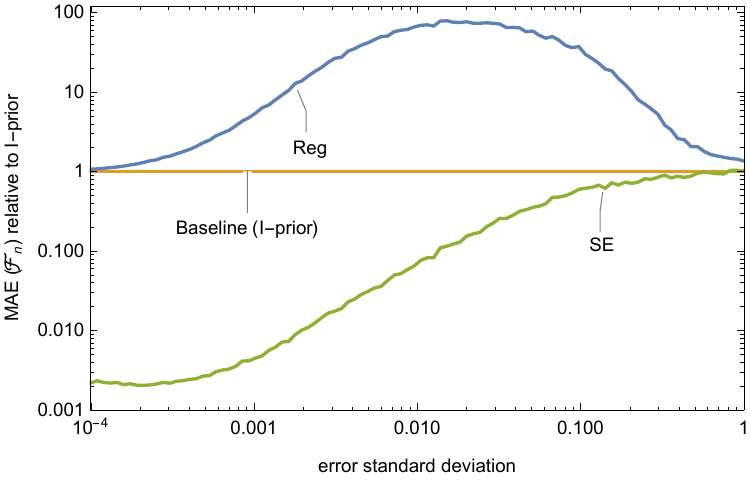}}}
	\\
	\caption{Panels on left: simulated MAE($\cF_n$) for Tikhonov regularizer (`Reg'), I-prior estimator (`I-prior'), and SE estimator (`SE'). The baseline level is the MAE if the zero function is fitted.  
		Panels on right: ratio of MAE($\cF_n$) for regularizer and SE estimator compared to I-prior.
		Model~(1) in the main article is assumed with $\cF$ the FBM-1/2 RKHS and i.i.d.\  normal errors. }
	\label{fig-sim1}
\end{figure}

\begin{figure}[tbp]
	\centering
	\subfigure[True regression function has regularity 1]
	{\hspace*{0mm}\includegraphics[width=70mm]{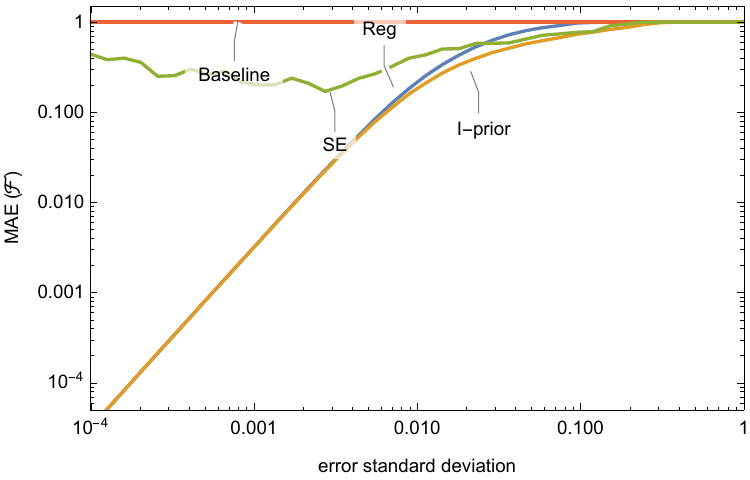}\hspace*{5mm}
		{\hspace*{0mm}\includegraphics[width=70mm]{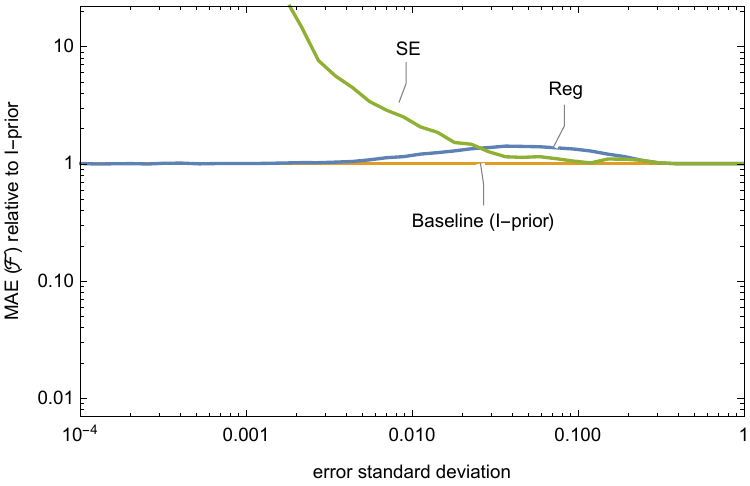}
	}}
	\\
	\subfigure[True regression function has regularity 3/2]
	{\hspace*{0mm}\includegraphics[width=70mm]{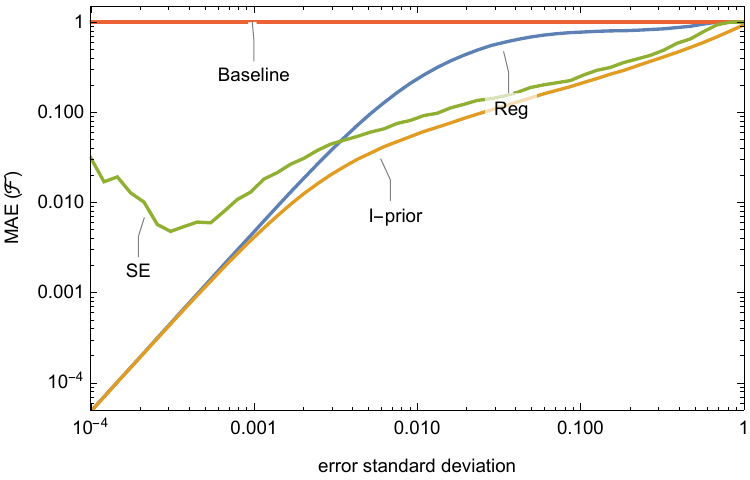}\hspace*{5mm}
		{\hspace*{0mm}\includegraphics[width=70mm]{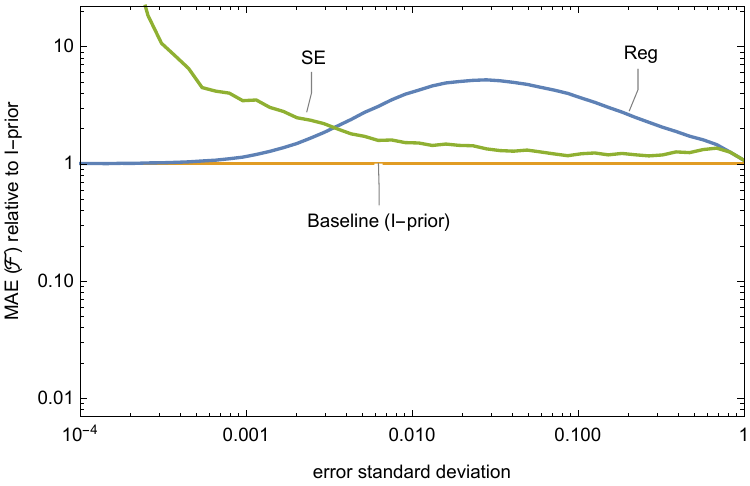}}}
	\\
	\subfigure[True regression function is a squared exponential Gaussian process path]
	{\hspace*{0mm}\includegraphics[width=70mm]{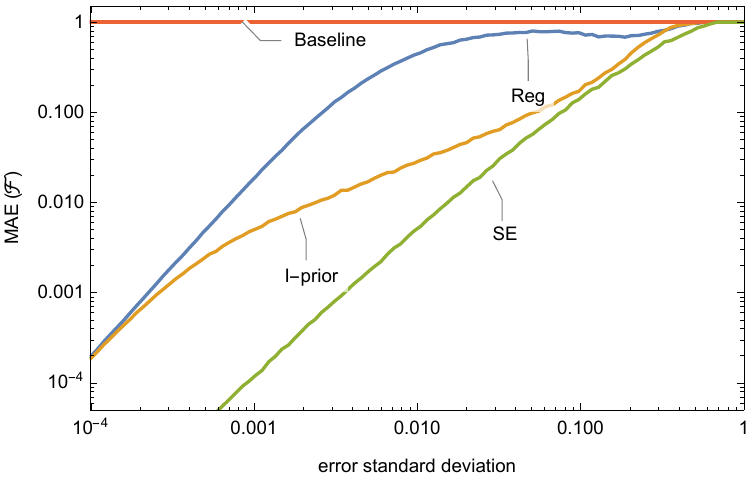}\hspace*{5mm}
		{\hspace*{0mm}\includegraphics[width=70mm]{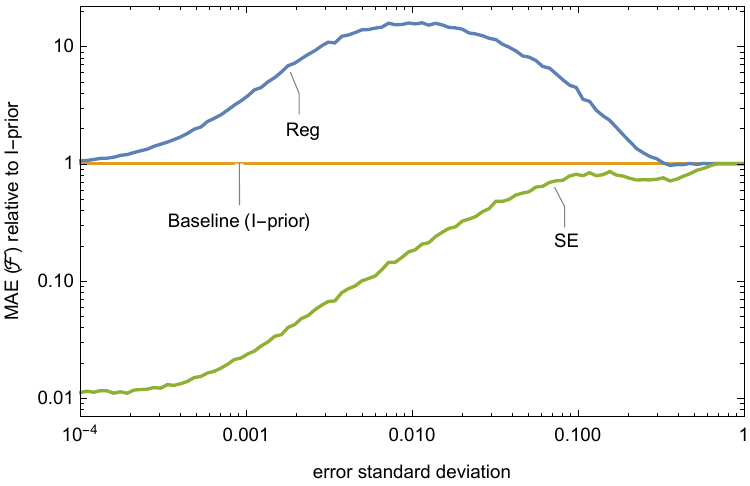}}}
	\\
	\caption{Panels on left: simulated MAE($\cF$) for Tikhonov regularizer (`Reg'), I-prior estimator (`I-prior'), and SE estimator (`SE'). The baseline level is the MAE if the zero function is fitted.  
		Panels on right: ratio of MAE($\cF$) for regularizer and SE estimator compared to I-prior.
		Model~~(1) in the main article is assumed with $\cF$ the FBM-1/2 RKHS and i.i.d.\  normal errors. }
	\label{fig-sim2}
\end{figure}

